\documentclass{amsart}

\usepackage[margin=1in]{geometry}
\usepackage{amssymb}
\usepackage[hidelinks]{hyperref}
\usepackage[capitalize]{cleveref}
\usepackage{microtype}

\newcommand{\R}{\mathbf{R}}
\newcommand{\Z}{\mathbf{Z}}
\newcommand{\N}{\mathbf{N}}

\newcommand{\E}{\mathbf{E}}

\newcommand{\abs}[1]{\lvert #1 \rvert}
\newcommand{\norm}[1]{\lVert #1 \rVert}

\newtheorem{thm}{Theorem}

\newtheorem{lem}[thm]{Lemma}
\newtheorem{prop}[thm]{Proposition}

\theoremstyle{definition}

\theoremstyle{remark}

\newcommand{\Span}{\operatorname{Span}}
\newcommand{\rank}{\operatorname{rank}}

\begin{document}

\title{Spherical configurations over finite fields}
\author{Neil Lyall, \'{A}kos Magyar, Hans Parshall}
\thanks{The first author was partially supported by NSF-DMS grant 1702411 and Simons Foundation Collaboration Grant for Mathematicians 245792. The second author was partially supported by Grants NSF-DMS 1600840 and  ERC-AdG 321104.}

\address{Department of Mathematics, The University of Georgia, Athens, GA 30602, USA}
\email{lyall@math.uga.edu}
\address{Department of Mathematics, The University of Georgia, Athens, GA 30602, USA}
\email{magyar@math.uga.edu}
\address{Department of Mathematics, The Ohio State University, Columbus, OH 43210, USA}
\email{parshall.6@osu.edu}

%\keywords{}

\begin{abstract}
%The aim of this article is to show that an analogue of Graham's conjecture holds in finite field geometries for 4-point spherical configurations spanning two dimensions, and more generally for
We establish that if $d \geq 2k + 6$ and $q$ is odd and sufficiently large with respect to $\alpha \in (0,1)$, then every set $A\subseteq \mathbf{F}_q^d$ of size $|A| \geq \alpha q^d$ will contain an isometric copy of every spherical $(k+2)$-point configuration that spans $k$ dimensions.
\end{abstract}
\maketitle

\setlength{\parskip}{3pt}

\section{Introduction}

Geometric Ramsey theory has its origins in series of papers by Erd\H{o}s et al. \cite{egmrss,egmrss2,egmrss3}, where they studied geometric configurations which cannot be destroyed by partitioning Euclidean space into finitely many classes. The fundamental problem is to classify those finite sets $X$ which are \emph{Ramsey}, in the sense that for every number of colors $r\in\N$ there is a dimension $d = d(r,X)$ for which every $r$-coloring of $\R^d$ contains a monochromatic, congruent copy of $X$. 

The simplest example of a Ramsey set is a regular $k$-simplex; that is, $k + 1$ equidistant points.  Indeed, for any dimension $d \geq kr$, any $r$-coloring of a regular $d$-simplex contains $k + 1$ points of the same color, forming a monochromatic regular $k$-simplex. On the other hand, a simple construction using the geometry of the Euclidean metric shows that any set of three collinear points is not Ramsey.  In fact, Erd\H{o}s et al. \cite{egmrss} showed that every Ramsey set must be \emph{spherical}; that is, contained in some sphere. This has led to the conjecture by Graham \cite{graham94} that a finite set $X$ is Ramsey if and only if it is spherical.

This conjecture is far from settled.  Examples of sets known to be Ramsey include vertices of ``bricks'' ($k$-dimensional rectangles) \cite{egmrss}, non-degenerate simplices \cite{frankl90}, trapezoids \cite{kriz92}, regular polygons and regular polyhedra~\cite{kriz91}. Common to many of these results was the exploitation of additional symmetries of the configuration. It was observed by Leader, Russell and Walters \cite{leader12} that all known examples of Ramsey sets are \emph{subtransitive} in the sense that they can be embedded in a higher dimensional set on which the rotation group acts transitively.  They introduced a rival conjecture that a finite set $X$ is Ramsey if and only if it is subtransitive, and further showed \cite{leader11} that almost all 4-point subsets of a circle are not subtransitive.  This was later extended by Eberhard \cite{eberhard13} to show that almost all $(k + 2)$-point sets on the $(k - 1)$-sphere are not subtransitive.  It remains an open question whether or not such configurations are Ramsey.

The aim of this article is to show that an analogue of Graham's conjecture holds in finite field geometries for 4-point spherical configurations spanning two dimensions, and more generally for spherical $(k+2)$-point configurations spanning $k$ dimensions.  We in fact prove a stronger \emph{density} version; namely that if $d \geq 2k + 6$ and $q$ is taken to be odd and sufficiently large with respect to $\alpha \in (0,1)$, then every set $A\subseteq \mathbf{F}_q^d$ of size $|A| \geq \alpha q^d$ contains an isometric copy of every such configuration $X$.  To be clear, here we say that two sets $X$ and $X'$ are \emph{isometric} if there is a bijection $\phi:X\to X'$ such that $|\phi(x)-\phi(x')|^2=|x-x'|^2$ for all $x \in X$ and $x'\in X'$, where $|x|^2=x\cdot x$ is the usual dot product of the vector $x\in \mathbf{F}_q^d$ with itself.

Our approach takes the point of view modern arithmetic combinatorics which has been very successful in the study of linear patterns in subsets of $\Z$ of positive density \cite{gowers98, gowers2001new}. In fact, one of the main purpose of this article is to extend these techniques to the setting of geometric Ramsey theory, where one counts configurations determined by both linear as well as certain non-linear relations, i.e. isometries.

The setting of vector spaces over finite fields provides a useful model to study many problems in arithmetic combinatorics; see especially the surveys \cite{green05,shparlinski13, wolf15}. In the context of geometric Ramsey theory over finite fields, notable results have been obtained by a number of authors \cite{chapman12,hart-iosevich,iosevich-rudnev,ff-simplices}. However, those results concern patterns consisting of points in general position, with no linear relations between them, and hence are fundamentally different.

Analogous results for simplices in Euclidean spaces and the integer lattice have been given in \cite{bourgain,magyar}, and it reasonable to expect that our approach here may be successfully adapted to these settings.  In the context of geometric (density) Ramsey theory in $\R^d$, some results using this approach were recently
obtained in \cite{cook-magyar-pramanik, lyall-magyar}. We hope to address further adaptation in the near future.

\subsection{Outline of paper}
The main results of the paper are stated is \cref{MainResults} below, and some preliminaries and reductions are presented in Sections \ref{preliminaries} and \ref{reductionSection}.

A key observation of the paper, see \cref{gvn} in \cref{TwoProps}, is that the count of isometric copies of a fixed configuration $X$ along bounded functions is controlled by a certain uniformity norm.  This norm, which we introduce in \cref{TwoProps}, measures the uniformity or randomness of a function along geometric rectangles, and it should be compared with the so-called $U^2$-uniformity norm of Gowers \cite{gowers98} which measures uniformity along combinatorial rectangles.  If a set $A$ is sufficiently uniform with respect to this norm, then it quickly follows from \cref{gvn} that $A$ contains many, in fact the statistically correct number of, isometric copies of $X$. The proof of \cref{gvn} is presented in \cref{gvnproofSection}.

In order to handle arbitrary sets $A$ we prove an \emph{inverse theorem}, see \cref{inverse} in \cref{inverseSection}, to establish that functions with large uniformity norm correlate with structured sets. Given such an inverse theorem there are then various standard  iterative procedures that one may hope to adapt to this setting to complete the argument. We follow an \emph{energy increment} route which to leads a so-called \emph{arithmetic regularity lemma}, namely \cref{decomposition}, that allows us to decompose the indicator function of $A$ as $1_A=f_1+f_2+f_3$, where $f_1$ is highly structured, $f_2$ has small $L^2$ norm and $f_3$ has small uniformity norm.  The proof of \cref{decomposition} is presented in \cref{last}.

In Sections \ref{5.1} and \ref{5.2} we demonstrate how \cref{decomposition} leads to a proof of our main results.  This consists of counting the isometric copies of $X$ along the main term $f_1$ and showing the contribution of the functions $f_2$ and $f_3$ are negligible. The setback of this approach is that it leads to very weak bounds, in fact the dependence of $q$ on $\alpha$ is tower-exponential. It seems quite possible that one could instead proceed via a density increment argument and obtain better, exponential type bounds, but we do not pursue this here.

We conclude the paper with an appendix in which we discuss the necessity of the spherical condition in the statement of our main result.

\section{Main results}\label{MainResults}

We will always work with a finite field $\mathbf{F}_q$ of odd characteristic.  For vectors $v,w \in \mathbf{F}_q^d$, we define their dot product $v \cdot w := \sum_{j = 1}^d v_jw_j$ as usual and we will work with the isotropic measurements of \emph{length} $|v|^2 := v \cdot v$ and \emph{distance} $|v - w|^2$.  For any $u \in \mathbf{F}_q^d$ and $\lambda \in \mathbf{F}_q$, we define the \emph{sphere} $S_\lambda(u) = \{x \in \mathbf{F}_q^d : |x - u|^2 = \lambda\}$, and we will simply write $S_\lambda$ when $u = 0$.  For $k \in \mathbf{N}$, we will say that $X \subseteq \mathbf{F}_q^d$ \emph{spans $k$ dimensions} when $\dim(\Span(X - X)) = k$, and we will call $X$ \emph{spherical} provided exists a sphere $S_\lambda(u) \subseteq \mathbf{F}_q^d$ with $X \subseteq S_\lambda(u)$.  Note the $(k + 1)$-point spherical configurations spanning $k$ dimensions are exactly the $k$-simplices, which were shown in \cite{ff-simplices} to appear as isometric copies in sufficiently dense subsets of $\mathbf{F}_q^d$ provided $d > k$.

In Euclidean spaces it is easy to see that if the finite sets $X$ and $X'$ are isometric, then $X'=z+U(X)$
for some vector $z$ and orthogonal transformation $U\subseteq O(d)$, i.e. $X'$ can be obtained from $X$ by a rigid motion.
The same may not hold in finite field geometries due to the presence of self-orthogonal vectors $x$ for which $|x|^2=0$.
However, it follows from Witt's extension theorem \cite{witt37} that if the subspace $V:=\Span(X-X)$ is non-degenerate in
the sense that $V\cap V^\perp =\{0\}$, then isometric copies of $X$ are indeed obtained by rigid motions. We use this fact
in the appendix to construct dense subsets avoiding isometric copies of non-spherical sets, establishing the necessity of
restricting our attention to spherical configurations $X$.

The main result of this paper is the following.

\begin{thm}\label{maintheorem}
	Let $d,k \in \mathbf{N}$ with $d \geq 2k + 6$, $\alpha \in (0,1)$, and $q \geq q(\alpha,k)$.  If $A \subseteq \mathbf{F}_q^d$ with $|A| \geq \alpha q^d$, then $A$ contains at least $c(\alpha,k)q^{(k + 1)d - k(k + 1)/2}$ isometric copies of every $(k + 2)$-point spherical configuration spanning $k$ dimensions.
\end{thm}

Here we write $c(\alpha,k)$ to stand in for some positive constant depending only on $\alpha$ and $k$, and $q \geq q(\alpha,k)$ to indicate $q$ is taken sufficiently large with respect to $\alpha$ and $k$.  We will use similar notation to indicate the dependency of constants $c, C > 0$ that may change between occurrences.  It is helpful to think of $\alpha$ and $k$ as fixed with $q$ allowed to tend toward infinity, and implicit constants in our big-O notation may depend on $k$.

Note that if $A$ would be a random subset of $\mathbf{F}_q^d$ of density $\alpha$ then it would contain $\alpha^{k+2}q^{(k + 1)d - k(k + 1)/2}$ isometric copies of a $(k + 2)$-point spherical configuration $X$ up to an error. This is because there are $k(k+1)/2$ quadratic relations, given by the length of the edges, between the points of the configuration $X$ and each vertex is contained in $A$ with probability $\alpha$. 

%We call a finite set $X\subseteq\mathbf{F}_q^d$ \emph{non-degenerate} if $V\cap V^\bot=\{0\}$ for $V=Span (X-X)$. As indicated above (see the appendix) a set $Y$ is isometric to $X$ if and only if $Y=z+U(X)$ for some vector $z$ and an orthogonal transformation $U$, to which we refer to as $Y$ being \emph{congruent} to $X$. Thus we have the following
%
%\begin{cor} Let $d,k \in \mathbf{N}$ with $d \geq 2k + 6$, $\alpha \in (0,1)$, and $q \geq q(\alpha,k)$.  If $A \subseteq \mathbf{F}_q^d$ with $|A| \geq \alpha q^d$, then $A$ contains at least $c(\alpha,k)q^{(k + 1)d - \binom{k + 1}{2}}$ congruent copies of every non-degenerate $k$-spherical configuration $X$ with $|X| = k + 2$.
%\end{cor}
%
%Let us remark that Graham's conjecture is expressed in this form, and also that two formulations coincide for generic spherical configurations (see the appendix).

It immediately follows from \cref{maintheorem} that for any fixed number of colors $r$, if we take $q$ sufficiently large with respect to $r$, then any $r$-coloring of $\mathbf{F}_q^{10}$ contains monochromatic, isometric copies of all 4-point spherical sets spanning 2 dimensions; this establishes a finite field version of Graham's conjecture that all cyclic quadrilaterals are Ramsey.  In fact, we prove a stronger statement relative to spheres; see \cite{graham83,matousek95} for some so-called ``sphere Ramsey'' results in the Euclidean setting.  

\begin{thm}\label{spheretheorem}
	Let $d,k \in \mathbf{N}$ with $d \geq 2k + 6$, $\alpha \in (0,1)$, and $q \geq q(\alpha,k)$.  If $\lambda \in \mathbf{F}_q^*$ and $A \subseteq S_\lambda$ with $|A| \geq \alpha q^{d - 1}$, then $A$ contains at least $c(\alpha,k)q^{(k + 1)d - (k+1)(k + 2)/2}$ isometric copies of every $(k + 2)$-point spherical configuration spanning $k$ dimensions.
\end{thm}

A straightforward counting argument reveals that \cref{spheretheorem} quickly implies \cref{maintheorem}, this argument  is presented in \cref{reductionSection} below.

We remark here that the relationship between $d$ and $k$ in both theorems could be improved if one were only interested in ``high rank'' configurations; see the comments following \cref{mainAsymptotic}.  Our methods are further able to prove a version of \cref{spheretheorem} when $\lambda = 0$ provided $d > 2k + 6$, but we do not pursue this since it does not impact our proof of \cref{maintheorem}.

\section{Preliminaries}\label{preliminaries}

Here we record notation and ingredients that we will require for the proof of \cref{spheretheorem}.  Given any function $f : \mathbf{F}_q^d \rightarrow \mathbf{C}$ and $B \subseteq \mathbf{F}_q^d$, we write
\[
	\E_{x \in B} f(x) := \frac{1}{|B|} \sum_{x \in B} f(x)
\]
for the average of $f$ over $B$, and we will understand the average $\E_x f(x)$ is taken over $\mathbf{F}_q^d$.  We will condense multiple averages $\E_{y_1}\E_{y_2} \cdots \E_{y_k}$ as $\E_{y_1, \ldots, y_k}$, and to indicate linear independence we will use the notation
\[
	\E_{y_1, \ldots, y_k}^* := \frac{1}{q^{kd}} \sum_{\substack{y_1, \ldots, y_k \in \mathbf{F}_q^d \\ \text{linearly independent}}}
\]
Letting $\chi$ denote the canonical additive character of $\mathbf{F}_q$, we define the Fourier transform $\widehat{f} : \mathbf{F}_q^d \rightarrow \mathbf{C}$ by
\[
	\widehat{f}(\xi) := \E_x f(x)\chi(-\xi \cdot x)
\]
and we recall the Fourier inversion formula
$f(x) = \sum_{\xi \in \mathbf{F}_q^d} \widehat{f}(\xi)\chi(\xi \cdot x).
$
Given two functions $f,g : \mathbf{F}_q^d \rightarrow \mathbf{C}$, we recall Plancherel's identity
\[
	\E_x f(x)\overline{g(x)} = \sum_{\xi \in \mathbf{F}_q^d} \widehat{f}(\xi)\overline{\widehat{g}(\xi)}
\]
and, defining the convolution $f * g(x) = \E_y f(y)g(x - y)$, we also recall $\widehat{f * g}(\xi) = \widehat{f}(\xi)\widehat{g}(\xi)$.

We will write $\sigma_\lambda = q1_{S_\lambda}$ for a normalized indicator function of $S_\lambda$, where we have the asymptotic
\begin{equation}\label{sphere}
	\widehat{\sigma}_\lambda(\xi) = \begin{cases} 1 + O(q^{-1/2}) & \text{if } \xi = 0 \\ O(q^{-1/2}) & \text{otherwise} \end{cases}
\end{equation}
valid for $d \geq 2$ and $\lambda \in \mathbf{F}_q^*$; see, for instance, \cite[Lemma 2.2]{iosevich-rudnev} exploiting Weil's bounds on Kloosterman sums.  

To vectors $y_1, \ldots, y_{j - 1} \in \mathbf{F}_q^d$ and constants $c_1, \ldots, c_j \in \mathbf{F}_q$, we will associate the \emph{spherical measure}
\begin{equation}\label{jstar}
	\sigma_{y_1, \ldots, y_{j - 1}}^{c_1, \ldots, c_j}(y_j) = \begin{cases} q^j & \text{if } y_i \cdot y_j = c_i \text{ for all } 1 \leq i \leq j \\ 0 & \text{otherwise} \end{cases}.
\end{equation}
This is essentially an $L^1$-normalized indicator function for the intersection of the sphere $S_{c_j}$ with $j - 1$ hyperplanes, so one should expect Fourier decay in appropriate directions.  To import the corresponding Fourier asymptotics, we set
\[
	\delta_{y_1, \ldots, y_{j-1}}(\xi) := \begin{cases} 1 & \text{if } \xi \in \Span(y_1, \ldots, y_{j - 1}) \\ 0 & \text{otherwise,} \end{cases}
\]
and record the simplest case of \cite[Lemma 7]{ff-simplices}.

\begin{lem}\label{mainAsymptotic} Let $c_1, \ldots, c_j \in \mathbf{F}_q$ with $c_j \neq 0$ and $y_1, \ldots, y_{j - 1} \in \mathbf{F}_q^d$ linearly independent.  If $d \geq 2j$, then
\[
	|\widehat{\sigma}_{y_1, \ldots, y_{j - 1}}^{c_1, \ldots, c_j}(\xi)| = \delta_{y_1, \ldots, y_{j - 1}}(\xi) + O(q^{-1/2}).
\]
\end{lem}
It is worth mentioning if one is willing to impose technical conditions on the defining vectors and scalars, then results within \cite{ff-simplices} include stronger asymptotics valid in the range $d > j$.  These would allow us to improve the required relationship of $d \geq 2k + 6$ in \cref{spheretheorem} for configurations avoiding self-orthogonal subspaces, but we opt instead for a uniform result valid for all $(k + 2)$-point spherical configurations spanning $k$ dimensions.  For clarity of presentation, we will often suppress the scalars $c_1, \ldots, c_j$ in the notation \eqref{jstar} since we will always restrict ourselves to $c_j \neq 0$ and \cref{mainAsymptotic} does not depend on the other scalars.

We will also require the notion of Bohr sets, which provide a substitute for fine subgroup structure.  We define the Bohr set of spectrum $\Gamma \subseteq \mathbf{F}_q^d$ and radius $\rho \in (0,1]$ by
\begin{equation}\label{bohrsetdefinition}
	B(\Gamma,\rho) := \{x \in \mathbf{F}_q^d : |\chi(\xi \cdot x) - 1| \leq \rho \text{ for all } \xi \in \Gamma\}.
\end{equation}
Setting $\beta = |B(\Gamma,\rho)|/q^d$, we record the standard bound
\begin{equation}\label{bohrsetsize}
	\beta \geq \Big(\frac{\rho}{2\pi}\Big)^{|\Gamma|}
\end{equation}
which can be found in \cite{taoVu}.  Setting $B = B(\Gamma,\rho)$, we will use the $L^1$ normalized indicator function $\mu_B = \beta^{-1} 1_B$ and its repeated convolution $\nu_B = \mu_B * \mu_B * \mu_B * \mu_B$.  The structure provided by repeated convolution will frequently be useful.  For example, provided $d \geq 2$, we can apply Fourier inversion, \eqref{sphere} and Plancherel to see
\begin{equation}\label{twoconvolutions}
		\mu_B * \mu_B * \sigma(x) = \sum_{\xi \in \mathbf{F}_q^d} \widehat{\mu}_B(\xi)^2 \widehat{\sigma}(\xi) \chi(\xi \cdot x) = 1 + O\Big(q^{-1/2} \sum_{\xi \in \mathbf{F}_q^d} |\widehat{\mu}_B(\xi)|^2\Big) = 1 + O(\beta^{-1}q^{-1/2}),
\end{equation}
from which it immediately follows that
\begin{equation}\label{fourconvolutions}
	\nu_B * \sigma(x) = 1 + O(\beta^{-1}q^{-1/2}).
\end{equation}
In other words, provided $q$ sufficiently large with respect to $\beta$, repeated convolutions of $\mu_B$ with $\sigma$ can be considered essentially constant.

We call the Bohr set $B(\Gamma,\rho)$ regular if for every $\epsilon > 0$, we have both
\begin{align*}
	|B(\Gamma,(1 + \epsilon)\rho)| &\leq (1 + 100\epsilon|\Gamma|) |B(\Gamma,\rho)|\\
	|B(\Gamma,(1 - \epsilon)\rho)| &\geq (1 - 100\epsilon|\Gamma|) |B(\Gamma,\rho)|;
\end{align*}
This definition, due to Bourgain \cite{bourgain99}, ensures that all Bohr sets are only a small dilation away from being regular (see \cite[Lemma 4.25]{taoVu}), and regular Bohr sets are essentially closed under addition by elements of their interior.   Given a regular Bohr set $B = B(\Gamma,\rho)$, we will write $B' \prec_\epsilon B$ in the case that $B' = B(\Gamma',\rho')$ is another regular Bohr set with $\Gamma' \supseteq \Gamma$ and $\rho' \leq \epsilon\rho / (200 \abs{\Gamma})$.  The following standard lemma provides the main consequence of regularity for our purposes; we include the proof for completeness.

\begin{lem}\label{regularity}
	Let $\epsilon \in (0,1)$, $B = B(\Gamma,\rho)$ regular with $B' \prec_\epsilon B$, and $f : \mathbf{F}_q^d \rightarrow \mathbf{C}$ with $|f| \leq 1$.  For any $y \in B'$,
	\[
		\abs{\E_{x \in B} f(x) - \E_{x \in B} f(x + y)} \leq \epsilon.
	\]
\end{lem}

\begin{proof}
	We have
	\[
		\abs{\E_{x \in B} f(x)  - \E_{x \in B} f(x + y)} = \frac{1}{|B|}\Big|\sum_{x \in \mathbf{F}_q^d} f(x)(1_B(x) - 1_B(x - y))\Big|
		\leq \frac{|B \triangle (y + B)|}{|B|}
	\]
	Since $B' \prec_\epsilon B$ and $y \in B'$ we have the relationship
	\[
		B \triangle (y + B) \subseteq B(\Gamma, \rho + \rho') \setminus B(\Gamma, \rho - \rho')
	\]
	and our claim follows from the bound $|B(\Gamma,\rho + \rho') \setminus B(\Gamma, \rho - \rho')| \leq \epsilon|B|$ resulting from regularity.
\end{proof}

\section{Reduction to Dense Spherical Sets}\label{reductionSection}

In this section we present the straightforward counting argument that estblishes  \cref{maintheorem} as a consequence of \cref{spheretheorem}.  The main observation is that the collection of spheres of a fixed radius $\lambda \in \mathbf{F}_q^*$ provides a uniform cover of $\mathbf{F}_q^d$.  Hence, any subset $A \subseteq \mathbf{F}_q^d$ with density $\alpha$ has density nearly $\alpha$ on a large number, in fact a positive proportion, of these spheres.  \cref{spheretheorem} then implies that within each of these spheres, $A$ contains many of the sought after configurations.  By counting the contribution of each of these spheres, it is easy to see that $A$ contains a positive proportion of the count of all such configurations within $\mathbf{F}_q^d$, as each fixed configuration is contained in approximately the same number of spheres.

We first record the characterization of spherical configurations that will be most useful as we proceed.  This follows in a straightforward way from \cref{appendixspherical} in the Appendix.

\begin{lem}\label{sphericalcondition} Let $X \subseteq \mathbf{F}_q^d$ be any $(k + 2)$-point spherical configuration spanning $k$ dimensions.  If any $(k + 1)$-point subset of $X$ that spans $k$ dimensions is contained in a sphere $S_\lambda(u)$, then $X \subseteq S_\lambda(u)$ as well.
\end{lem}

Let us now see how \cref{spheretheorem} implies \cref{maintheorem}.

\begin{proof}[Proof of \cref{maintheorem}]
	Fixing $\lambda \in \mathbf{F}_q^*$, we first establish that for many centers $u \in \mathbf{F}_q^d$, $|A \cap S_\lambda(u)|$ is large.  For $d \geq 2$ and $x \in \mathbf{F}_q^d$, \eqref{sphere} implies
	$
		\E_u \sigma_\lambda(u - x) = 1 + O(q^{-1/2}),
	$	
	in which case we can ensure $q$ is sufficiently large for
	\begin{align*}
		(\alpha/2) q^{2d - 1} &\leq \sum_{u \in \mathbf{F}_q^d} |A \cap S_\lambda(u)|\\
		&= \sum_{\substack{u \in \mathbf{F}_q^d \\ |A \cap S_\lambda(u)| \geq (\alpha/4) q^{d - 1}}} |A \cap S_\lambda(u)| + \sum_{\substack{u \in \mathbf{F}_q^d \\ |A \cap S_\lambda(u)| < (\alpha/4) q^{d - 1}}} |A \cap S_\lambda(u)|\\
		&\leq |S_\lambda||\{u \in \mathbf{F}_q^d : |A \cap S_\lambda(u)| \geq (\alpha/4)q^{d - 1}\}| + (\alpha/4)q^{2d}.
	\end{align*}	
	Again using \eqref{sphere} and ensuring $q$ sufficiently large, we can conclude
	\[
		|\{u \in \mathbf{F}_q^d : |(A - u) \cap S_\lambda| \geq (\alpha/4)q^{d - 1}\}| \geq (\alpha/8)q^{d}.
	\]
	
	Fixing now some $(k + 2)$-point spherical configuration $X$ spanning $k$ dimensions, we see that for each of these good centers $u$, \cref{spheretheorem} guarantees that $A \cap S_\lambda(u)$ contains at least $c(\alpha,k)q^{(k + 1)d - \binom{k + 2}{2}}$ isometric copies of $X$.	We need to account for how many spheres we count each isometric copy within.  By translating $X$ if necessary, we can parametrize it as
	\[
		X = \{0, v_1, \ldots, v_k, a_1v_1 + \cdots + a_kv_k\}
	\]
	for linearly independent $v_1, \ldots, v_k \in \mathbf{F}_q^d$ and $a_1, \ldots, a_k \in \mathbf{F}_q$.  Then for a translation $x \in \mathbf{F}_q^d$ and linearly independent vectors $y_1, \ldots, y_k \in \mathbf{F}_q^d$ with $y_i \cdot y_j = v_i \cdot v_j$ for all $1 \leq i \leq j \leq k$, we wish to count how many spheres $S_\lambda(u)$ contain the configuration
	\[
	X' = \{x, x + y_1, \ldots, x + y_k, x + a_1y_1 + \cdots + a_ky_k\}.
	\]
	Since $y_1, \ldots, y_k$ are linearly independent, \cref{mainAsymptotic} applies to the spherical measure
	\[
		\sigma_{y_1, \ldots, y_k} (u) := \sigma_\lambda(u) \prod_{j = 1}^k \sigma_\lambda(y_j - u)
	\]
	so that, for $d \geq 2k + 2$,
	\[
		\E_{u \in \mathbf{F}_q^d} \sigma_{y_1, \ldots, y_k}(x - u) = 1 + O(q^{-1/2}).
	\]
	\cref{sphericalcondition} ensures that $\{x, x + y_1, \ldots, x + y_k\} \subseteq S_\lambda(u)$ implies $X' \subseteq S_\lambda(u)$ as well, so we have shown
	\[
		|\{u \in \mathbf{F}_q^d : X' \subseteq S_\lambda(u)\}| = (1 + O(q^{-1/2})) q^{d - k - 1}.
	\]
	Hence, within each of the $(\alpha/8)q^d$ good spheres $S_\lambda(u)$, $A \cap S_\lambda(u)$ contains $c(\alpha,k)q^{(k + 1)d - \binom{k + 2}{2}}$ isometric copies of $X$, and each of these copies is contained in roughly $q^{d - k - 1}$ spheres.  In total this yields that $A$ contains $c(\alpha,k)q^{(k + 1)d - \binom{k + 1}{2}}$ isometric copies of $X$ as claimed.
\end{proof}

\section{Proof of \cref{spheretheorem}}\label{mainargument}

In this section we reduce the task of proving \cref{spheretheorem} to that of establishing Propositions \ref{gvn} and \ref{decomposition} below.

\subsection{Counting Configurations, a Uniformity Norm, and two key Propositions}\label{TwoProps}

For the remainder, we fix $\lambda \in \mathbf{F}_q^*$ and aim to establish \cref{spheretheorem}; there is no harm in assuming $\lambda = 1$.  
For brevity, we set $S = S_\lambda$  and $\sigma = \sigma_\lambda$.  We will be considering $(k + 2)$-point spherical configurations spanning $k$ dimensions typically parameterized as
\[
	X := \{0, v_1, \ldots, v_k, a_1v_1 + \cdots + a_kv_k\} \subseteq \mathbf{F}_q^d,
\]
with $v_1, \ldots, v_k \in \mathbf{F}_q^d$ linearly independent and coefficients $a_1, \ldots, a_k \in \mathbf{F}_q$.  We will consider other sets of $k + 2$ points of the form
\[
	X' = \{x_0, x_0 + x_1, \ldots, x_0 + x_k, x_0 + a_1x_1 + \cdots + a_kx_k\} \subseteq \mathbf{F}_q^d
\]
and check whether these are indeed isometric copies of $X$ contained within $S$ by checking whether $x_1, \ldots, x_k$ are linearly independent and further satisfy the conditions
\begin{equation*}
|x_0|^2 = |x_0 + x_1|^2 = \cdots = |x_0 + x_k|^2 = \lambda \text{ and }
x_i \cdot x_j = v_i \cdot v_j \text{ for each }1 \leq i \leq j \leq k;
\end{equation*}
when all of this is true, we write $X' \simeq X$.  We remark that this notation only explicitly insists that $k + 1$ points of $X'$ lie on $S$.  However, since we are working with the same coefficients $a_1, \ldots, a_k$, \cref{appendixspherical} ensures $X'$ is spherical and \cref{sphericalcondition} ensures $X' \subset S$ as well.  That is, if $X' \simeq X$, then it is also the case that $X' \subseteq S$.  

To count copies of $X$ parameterized as $X'$ above, we define the weight
\[
	\mathcal{S}_X(x_0, \ldots, x_k) := \begin{cases} q^{(k+1)(k + 2)/2} & \text{if } X' \simeq X \\ 0 & \text{otherwise} \end{cases}
\]
and a normalized counting operator on functions $f_0, f_1, \ldots, f_{k + 1} : \mathbf{F}_q^d \rightarrow \mathbf{C}$ by
\begin{equation}\label{thecount}
	\mathcal{N}_X(f_0, \ldots, f_{k + 1}) := \E_{x_0} f_0(x_0)\E_{x_1, \ldots, x_k}^* \Big(\prod_{j = 1}^{k}f_j(x_0 + x_j)\Big)f_{k + 1}(x_0 + a_1x_1 + \cdots + a_kx_k)\mathcal{S}_X(x_0, \ldots, x_k).
\end{equation}

Note that as long as we restrict our attention to linearly independent $x_1, \ldots, x_k$, we can write out
\[
	\mathcal{S}_X(x_0, \ldots, x_k) = \sigma(x_0)\prod_{j = 1}^{k} \sigma_{x_0, \ldots, x_{j - 1}}(x_j)
\]
for appropriate spherical measures with implicit scalars determined by our sphere $S$ and the dot products between the defining vectors $v_1, \ldots, v_k$ of $X$.  Moreover, the contribution of linearly dependent $x_1, \ldots, x_k$ is negligible, in the sense that we trivially have
\[
	\frac{1}{q^d} \sum_{x_j \in \Span(x_1, \ldots, x_{j - 1})} \sigma_{x_1, \ldots, x_{j - 1}}(x_j) \leq q^{-1} \]
whenever $d \geq 2j$.  Together with \cref{mainAsymptotic}, this allows us to freely add in linearly dependent $x_1, \ldots, x_k$ to our count \eqref{thecount} at the cost of an acceptable $O(q^{-1})$ error, provided each $f_j$ is bounded as will typically be the case.

% TODO: Need better notation than \delta_{h \cdot h'} here...what did I use in thesis?

The starting point for our argument is to show that the counting operator \eqref{thecount} is controlled by what we call the $U^2_\perp(S)$ norm, defined for $f : \mathbf{F}_q^d \rightarrow \mathbf{C}$ by
\[
	\norm{f}_{U^2_\perp(S)} := \Big(\E_{x,h,h'} f\sigma(x)\overline{f}\sigma(x + h)\overline{f}\sigma(x + h')f\sigma(x + h + h')\Big)^{1/4},
\]
which is the usual Gowers $U^2(\mathbf{F}_q^d)$ norm of the function $f\sigma$.  While the $U^2(\mathbf{F}_q^d)$ norm averages over combinatorial rectangles, the $U^2_\perp(S)$ norm averages instead over geometric rectangles contained within our sphere $S$.  Note that if
\[
	|x|^2 = |x + h|^2 = |x + h'|^2 = |x + h + h'|^2 = \lambda,
\]
then it is also the case that $h \cdot h' = 0$.  

We will show that the operator \eqref{thecount} is controlled by the $U^2_\perp(S)$ norm in the following sense.

\begin{prop}\label{gvn}
	Let $f_0, \ldots, f_{k + 1} : \mathbf{F}_q^d \rightarrow \mathbf{C}$ with $|f_j| \leq 1$.  If $d \geq 2k + 6$, then
	\[
		|\mathcal{N}_X(f_0, \ldots, f_{k + 1})| \leq \min_{0 \leq j \leq k + 1} \norm{f_j}_{U^2_\perp(S)} + O(q^{-1/8})
	\]
\end{prop}

Results of this type are often called  \emph{generalized von-Neumann inequalities} in the arithmetic combinatorics literature. The proof of \cref{gvn} is presented in 
 \cref{gvnproofSection} below.  
 
 To see the utility of such a result, consider a set $A \subseteq S$ with $|A| \geq \alpha q^{d - 1}$ and $f_A = 1_A - \alpha$.  If $A$ is sufficiently uniform, in the sense that $\norm{f_A}_{U^2_\perp(S)}$ is sufficiently small with respect to $\alpha$ and $k$, then the decomposition $1_A = \alpha + f_A$ along with \cref{gvn} provides
\[
	\mathcal{N}_X(1_A, \ldots, 1_A) \gtrsim \alpha^{k + 2}
\]
provided $q$ is taken sufficiently large with respect to $\alpha$ and $k$.  Of course, not all sets must be uniform in this sense, and we will require a more sophisticated decomposition.  Defining the $L^2(S)$ norm by
\[
	\norm{f}_{L^2(S)} := \Big(\E_x |f(x)|^2\sigma(x)\Big)^{1/2},
\]
we will use the following decomposition.

\begin{prop}\label{decomposition} Let $\eta \in (0,1)$, $f : \mathbf{F}_q^d \rightarrow [-1,1]$ and $\varphi : (0,1]^2 \rightarrow (0,1)$ increasing in both coordinates.  If $d \geq 2$ and $q \geq q(\eta,f)$, then there exists $B = B(\Gamma,\rho)$ regular with $|\Gamma| \leq C(\eta,f)$ and $\rho \geq c(\eta,f)$, and there exist functions $f_1, f_2, f_3 : \mathbf{F}_q^d \rightarrow [-2,2]$ with
\begin{align*}
	f &= f_1 + f_2 + f_3\\
	f_1 &= \nu_B * (f\sigma)\\
	\norm{f_2}_{U^2_\perp(S)} &\leq \varphi(|\Gamma|^{-1},\rho)\\
	\norm{f_3}_{L^2(S)} &\leq \eta.
\end{align*}
\end{prop}

Results of this type are often called  \emph{arithmetic regularity lemmas} in the arithmetic combinatorics literature. 
In light of \eqref{fourconvolutions}, it should at least be clear that ensuring $f_1$ is bounded amounts to ensuring $q$ is sufficiently large with respect to the other parameters. 
The proof of \cref{decomposition} is presented in 
 \cref{last} below.
 
 \smallskip
 
   We close this section by demonstrating, as promised, how Propositions \ref{gvn} and \ref{decomposition} can be applied to give a proof of \cref{spheretheorem}.  We initially specialize,  in \cref{5.1} below, to the case of spherical quadrilaterals, that is when $k=2$. The proof of the general case follows along similar lines and is presented in \cref{5.2}.
   
\subsection{Proof that Propositions \ref{gvn} and \ref{decomposition} imply the $k=2$ case of \cref{spheretheorem}}\label{5.1}

It clearly suffices to establish the following
\begin{thm}
	Let $\alpha \in (0,1)$ and $A \subseteq S$ with $|A| \geq \alpha q^{d - 1}$.  If $d \geq 10$, $q \geq q(\alpha)$, and $X \subseteq \mathbf{F}_q^d$ is a spherical 4-point configuration spanning 2 dimensions, then
	\[
		\mathcal{N}_X(1_A, 1_A, 1_A, 1_A) \geq c(\alpha).
	\]
	In other words, $A$ contains $c(\alpha) q^{3d - 6}$ isometric copies of $X$.
\end{thm}

\begin{proof}
To set up, let $\epsilon > 0$ be a parameter to be determined only in terms of $\alpha$ and let $\varphi : (0,1]^2 \rightarrow (0,1)$ be a function increasing in both coordinates to be determined only in terms of $\epsilon$.  We apply the decomposition theorem to obtain a regular Bohr set $B = B(\Gamma,\rho)$ with $\abs{\Gamma} \leq C(\epsilon,\varphi)$ and $\rho \geq c(\epsilon,\varphi)$ and functions $f_1, f_2, f_3 : \mathbf{F}_q^d \rightarrow [-2,2]$ such that
\begin{align*}
	1_A &= f_1 + f_2 + f_3\\
	f_1 &= \nu_B * (1_A\sigma)\\
	\norm{f_2}_{U^2_\perp(S)} &\leq \varphi(|\Gamma|^{-1},\rho)\\
	\norm{f_3}_{L^2(S)} &\leq \epsilon.
\end{align*}
We take $\rho' \in [\epsilon \rho / (400|\Gamma|), \epsilon \rho / (200 |\Gamma|)]$ so that $B_1 := B(\Gamma,\rho') \prec_\epsilon B$.  Throughout the argument we will continue to take $q(\alpha)$ large enough so that, with parameters other than $q$ fixed, we can absorb error terms that tend to zero as $q \rightarrow \infty$ into a single $O(\epsilon)$ error.  We now fix a spherical 4-point configuration spanning 2 dimensions $X$, parameterize it as
\[
	X = \{0, v, w, av + bw\}
\]
for $v,w$ linearly independent, and search for isometric copies of the form
\[
	\{x, x + y, x + z, x + ay + bz\} \subseteq S
\]
where $|y|^2 = |v|^2$, $|z|^2 = |w|^2$, and $y \cdot z = v \cdot w$.  To detect these copies, we define the spherical measures
\begin{align*}
	\sigma_x(y) &:= \sigma_x^{-|v|^2/2,|v|^2}(y), \text{ and}\\
	\sigma_{x,y}(z) &:= \sigma_{x,y}^{-|w|^2/2, v \cdot w,|w|^2}(z)
\end{align*}
in which case we can parametrize our counting operator as
\[
	\mathcal{N}_{X}(g_0, g_1, g_2, g_3) = \E_{x} g_0(x)\sigma(x) \E_{y,z}^* g_1(x + y)g_2(x + z)g_3(x + ay + bz) \sigma(x)\sigma_x(y)\sigma_{x,y}(z).
\]
Setting
\[
	B_2 := B(\Gamma \cup \{a \cdot \xi : \xi \in \Gamma\} \cup \{b \cdot \xi : \xi \in \Gamma\}, \rho'/4),
\]
we note that restricting $y,z \in B_2 + B_2$ ensures $y,z,ay + bz \in B_1$.  This is useful to us, since for our main term function $f_1$, for any $x \in \mathbf{F}_q^d$ and any $x' \in B_1$, \cref{regularity} provides
\begin{equation}\label{bohrclosure1}
	f_1(x + x') = f_1(x) + O(\epsilon).
\end{equation}
We will in fact work with the further restricted count
\[
	\mathcal{N}_{X}^B(g_0, g_1, g_2, g_3) := \E_{x} g_0\sigma(x)\E_{y,z}^* g_1(x + y)g_2(x + z)g_3(x + ay + bz) \sigma_x(y)\mu_{B_2} * \mu_{B_2}(y)\sigma_{x,y}(z) \mu_{B_2} * \mu_{B_2}(z).
\]
Setting $\beta := |B_2|/q^d$, the size estimate \eqref{bohrsetsize} and the dependency of our parameters guarantees
\[
	\beta^{-1} \leq \left(\frac{C|\Gamma|}{\epsilon \rho}\right)^{3|\Gamma|} \leq C(\alpha).
\]
We should establish that this restricted count is well normalized.  That is, we will show
\begin{equation}\label{bohrCountNormalized1}
	\mathcal{N}_{X}^B(\mathbf{1}, \mathbf{1}, \mathbf{1}, \mathbf{1}) = 1 + O(\epsilon),
\end{equation}
where here $\mathbf{1}$ stands for the constant 1 function.  Applying Parseval and extracting the $\xi = 0$ term,
\[
	\E_z \mu_{B_2} * \mu_{B_2}(z) \sigma_{x,y}(z) = 1 + \sum_{\xi \in \mathbf{F}_q^d \setminus \{0\}} \widehat{\mu}_{B_2}(\xi)^2 \overline{\widehat{\sigma}_{x,y}(\xi)} + O(\epsilon),
\]
so we can apply \cref{mainAsymptotic} for
\[
	\mathcal{N}_{X}^B(\mathbf{1},\mathbf{1},\mathbf{1},\mathbf{1}) = \E_{x,y} \sigma(x)\sigma_{x}(y)\mu_{B_2} * \mu_{B_2}(y) + O\left(\beta^{-1}\sum_{\xi \in \mathbf{F}_q^d \setminus \{0\}} |\widehat{\mu}_{B_2}(\xi)|^2 \E_{x,y} \sigma(x)\sigma_x(y) \delta_{x,y}(\xi)\right) + O(\epsilon).
\]
%%%%% TODO:  Fix?  Continue with argument.
Without trying to be too careful, notice that uniformly in $\xi \neq 0$,
\[
	\E_{x,y} \sigma(x)\sigma_{x}(y)\delta_{x,y}(\xi) \leq q^3 \E_{x,y} \delta_{x,y}(\xi) = O(q^{-1}),
\]
so we can apply Plancherel and ensure $q$ is sufficiently large to conclude
\[
	\mathcal{N}_{X}^B(\mathbf{1},\mathbf{1},\mathbf{1},\mathbf{1}) = \E_{x,y} \sigma(x)\sigma_{x}(y)\mu_{B_2} * \mu_{B_2}(y) + O(\epsilon).
\]
Arguing similarly to eliminate the average in $y$, we conclude \eqref{bohrCountNormalized1}.  This restricted count will be useful for us since we trivially have the lower bound
\begin{equation}\label{triviallowerbound}
	\mathcal{N}_{X}(1_A, 1_A, 1_A, 1_A) \geq \beta^2  \mathcal{N}_{X}^B(1_A, 1_A, 1_A, 1_A),
\end{equation}
and we will spend the rest of the proof establishing the lower bound
\[
	\mathcal{N}^B_{X}(1_A, 1_A, 1_A, 1_A) \geq \alpha^4 + O(\epsilon).
\]
From the decomposition $1_A = f_1 + f_2 + f_3$,
\[
	\mathcal{N}^B_{X}(1_A, 1_A, 1_A, 1_A) = \sum_{1 \leq i_0, i_1, i_2, i_3 \leq 3} \mathcal{N}^B_{X}(f_{i_0}, f_{i_1}, f_{i_2}, f_{i_3}),
\]
and we handle separately the three cases (i) when each of $i_0, i_1, i_2, i_3$ equals 1, (ii) when one of $i_0, i_1, i_2, i_3$ equals 2, and (iii) when one of $i_0, i_1, i_2, i_3$ equals 3.  Case (i) will yield our main term of $\alpha^4$, and we will argue that cases (ii) and (iii) contribute a negligible $O(\epsilon)$ error.

For the first case, applying \eqref{bohrclosure1} and \eqref{bohrCountNormalized1} yields
\[
	\mathcal{N}^B_{X}(f_1, f_1, f_1, f_1) = \E_x f_1(x)^4\sigma(x) \E_{y,z} \mu_{B_2} * \mu_{B_2}(y) \sigma_x(y) \mu_{B_2} * \mu_{B_2}(z) \sigma_{x,y}(z) + O(\epsilon).
\]
Arguing as we did to establish \eqref{bohrCountNormalized1},
\[
	\mathcal{N}^B_{X}(f_1, f_1, f_1, f_1) = \E_x f_1(x)^4\sigma(x) + O(\epsilon).
\]
Then from H\"{o}lder's inequality,
\[
	\mathcal{N}_{X}^B(f_1, f_1, f_1, f_1) \geq (\E_x f_1(x)\sigma(x))^4 + O(\epsilon),
\]
and writing
\begin{align*}
	\E_x f_1(x)\sigma(x) &= \E_x \nu_B * 1_A(x) \sigma(x)\\
	&= \E_x 1_A(x) \nu_B * \sigma(x),
\end{align*}
we can apply \eqref{fourconvolutions} to conclude
\[
	\mathcal{N}_{X}^B(f_1, f_1, f_1, f_1) \geq \alpha^4 + O(\epsilon).
\]

For the second case, with one of $i_0, i_1, i_2, i_3 = 2$, we apply Fourier inversion to write
\[
	\mu_{B_2} * \mu_{B_2}(y)\mu_{B_2} * \mu_{B_2}(z) = \sum_{\xi_1, \xi_2 \in \mathbf{F}_q^d} \widehat{\mu}_{B_2}(\xi_1)^2\widehat{\mu}_{B_2}(\xi_2)^2\chi(-\xi_1 \cdot y)\chi(-\xi_2 \cdot z).
\]
Setting $\chi_\xi(x) = \chi(\xi \cdot x)$, this allows us to express
\[
	\mathcal{N}_{X}^B(f_{i_0},f_{i_1},f_{i_2},f_{i_3}) =  \sum_{\xi_1, \xi_2 \in \mathbf{F}_q^d} \widehat{\mu}_{B_2}(\xi_1)^2\widehat{\mu}_{B_2}(\xi_2)^2 \mathcal{T}_{X,S}(f_{i_0}\chi_{\xi_1 + \xi_2}, f_{i_1}\chi_{-\xi_1}, f_{i_2}\chi_{-\xi_2}, f_{i_3}).
\]
As $\norm{f\chi_\xi}_{U^2_\perp(S)} = \norm{f}_{U^2_\perp(S)}$, we can apply the triangle inequality, \cref{gvn} and Plancherel to conclude
\begin{align*}
	|\mathcal{N}_{X}^B(f_{i_0},f_{i_1},f_{i_2},f_{i_3})| &\leq \beta^{-2} \norm{f_2}_{U^2_\perp(S)} + O(\epsilon)\\
	&\leq \left(\frac{C|\Gamma|}{\epsilon\rho}\right)^{6|\Gamma|} \varphi(|\Gamma|^{-1},\rho) + O(\epsilon).
\end{align*}
We see that by taking
\[
	\varphi(|\Gamma|^{-1},\rho) = \epsilon\left(\frac{c\epsilon\rho}{|\Gamma|}\right)^{6|\Gamma|},
\]
we can ensure case (ii) contributes at most $O(\epsilon)$.

For the third case, we will assume $i_0 = 3$, but each case is similar by reindexing in $x_0$.  Applying the triangle inequality and using that $f_{i_1}, f_{i_2}, f_{i_3}$ are each at most 2 in absolute value,
\[
	\mathcal{N}_{X}^B(f_3, f_{i_1},f_{i_2},f_{i_3}) \leq 8 \E_{x} |f_3(x)| \sigma(x) \E_{y,z}^* \sigma_x(y)\mu_{B_2} * \mu_{B_2}(y) \sigma_{x,y}(z)\mu_{B_2} * \mu_{B_2}(z).
\]
Then arguing as we did to establish \eqref{bohrCountNormalized1} we have
\[
	\mathcal{N}_{X}^B(f_3,f_{i_1},f_{i_2},f_{i_3}) \leq 8 \E_{x} |f_3(x)| \sigma(x) + O(\epsilon).
\]
Applying Cauchy-Schwarz and the fact that $\norm{f_3}_{L^2(S)} \leq \epsilon$, we conclude case (iii) again contributes at most an $O(\epsilon)$ error term.  In total we have shown
\[
	\mathcal{N}^B_{X}(1_A, 1_A, 1_A, 1_A) \geq \alpha^4 + O(\epsilon)
\]
which we can insist is at least $\alpha^4/2$ by taking $\epsilon$ sufficiently small with respect to $\alpha$.  Then since $\beta \geq c(\alpha)$, we are done by recalling \eqref{triviallowerbound}.
\end{proof}

\subsection{Proof that Propositions \ref{gvn} and \ref{decomposition} imply  \cref{spheretheorem} in its full generality}\label{5.2}

It clearly suffices to establish the following

\begin{thm}
	Let $\alpha \in (0,1)$ and $A \subseteq S$ with $|A| \geq \alpha q^{d - 1}$.  If $d \geq 2k + 6$, $q \geq q(\alpha,k)$, and $X \subseteq \mathbf{F}_q^d$ is a spherical $(k + 2)$-point configuration spanning $k$ dimensions, then
	\[
		\mathcal{N}_X(1_A, \ldots, 1_A) \geq c(\alpha,k).
	\]
	In other words, $A$ contains $c(\alpha,k)q^{(k + 1)d - (k+1)(k + 2)/2}$ isometric copies of $X$.
\end{thm}

\begin{proof}
The proof is essentially the same as for 4-point configurations spanning 2 dimensions, although the notation becomes more cumbersome.  As before, to set up, let $\epsilon > 0$ be a parameter to be determined only in terms of $\alpha$ and $k$ and let $\varphi : (0,1]^2 \rightarrow (0,1)$ be a function increasing in both coordinates to be determined only in terms of $\epsilon$.  We apply the decomposition theorem to obtain a regular Bohr set $B = B(\Gamma,\rho)$ with $\abs{\Gamma} \leq C(\epsilon,\varphi)$ and $\rho \geq c(\epsilon,\varphi)$ and functions $f_1, f_2, f_3 : \mathbf{F}_q^d \rightarrow [-2,2]$ such that
\begin{align*}
	1_A &= f_1 + f_2 + f_3\\
	f_1 &= \nu_B * (1_A\sigma)\\
	\norm{f_2}_{U^2_\perp(S)} &\leq \varphi(|\Gamma|^{-1},\rho)\\
	\norm{f_3}_{L^2(S)} &\leq \epsilon.
\end{align*}
We take $\rho' \in [\epsilon \rho / (400|\Gamma|), \epsilon \rho / (200 |\Gamma|)]$ so that $B_1 := B(\Gamma,\rho') \prec_\epsilon B$.  Throughout the argument we will continue to take $q(\alpha,k)$ large enough so that, with parameters other than $q$ fixed, we can absorb error terms that tend to zero as $q \rightarrow \infty$ into a single $O(\epsilon)$ error.  We now fix a spherical $(k + 2)$-point configuration spanning $k$ dimensions $X$ parameterized as
\[
	X = \{0, v_1, \ldots, v_{k + 1}\}
\]
where $v_{k + 1} = a_1v_1 + \cdots + a_kv_k$ and search for isometric copies of the form
\[
	\{x_0, x_0 + x_1, \ldots, x_0 + x_{k + 1}\} \subset S
\]
where $x_{k + 1} = a_1x_1 + \cdots + a_kx_k$ and $x_i \cdot x_j = v_i \cdot v_j$ for $1 \leq i, j \leq k$.  To detect these copies, we define the spherical measures
\[
	\sigma_{x_0, \ldots, x_{j - 1}}(x_j) := (\sigma_{})^{-|v_j|^2/2, v_1 \cdot v_j, \ldots, v_{j - 1} \cdot v_j,|v_j|^2}_{x_0, \ldots, x_{j - 1}}(x_j)
\]
in which case we can parametrize our count as
\[
	\mathcal{N}_{X}(g_0, \ldots, g_{k + 1}) = \E_{x_0} g_0(x_0)\sigma(x_0) \E_{x_1, \ldots, x_k}^* g(x_0 + x_{k + 1}) \prod_{j = 1}^k g(x_0 + x_j) \sigma_{x_0, \ldots, x_{j - 1}}(x_j).
\]
Setting
\[
	B_2 := B\left(\Gamma \cup \bigcup_{j = 1}^k \left\{a_j \cdot \xi : \xi \in \Gamma\right\}, \dfrac{\rho'}{2k}\right),
\]
ensures that if $x_1, \ldots, x_k \in B_2 + B_2$, then $x_1, \ldots, x_{k + 1} \in B_1$.  Again we will use that for $x \in \mathbf{F}_q^d$ and $x' \in B_1$,
\begin{equation}\label{bohrclosure2}
	f_1(x + x') = f_1(x) + O(\epsilon).
\end{equation}
Our restricted count is given by
\[
	\mathcal{N}_{X}^B(g_0, \ldots, g_{k + 1}) = \E_{x_0} g_0(x_0)\sigma(x_0) \E_{x_1, \ldots, x_k}^* g(x_0 + x_{k + 1}) \prod_{j = 1}^k g(x_0 + x_j) \mu_{B_2} * \mu_{B_2}(x_j)\sigma^{(j + 1)}(x_j) .
\]
As before, setting $\beta := |B_2|/q^d$, we have $\beta^{-1} \leq C(\alpha,k)$ and a well-normalized restricted count
\begin{equation}\label{bohrCountNormalized2}
	\mathcal{N}_{X}^B(\mathbf{1}, \ldots, \mathbf{1}) = 1 + O(\epsilon).
\end{equation}
Applying the straightforward lower bound
\begin{equation}\label{triviallowerbound2}
	\mathcal{N}_{X}(1_A, \ldots, 1_A) \geq \beta^k  \mathcal{N}_{X}^B(1_A, \ldots, 1_A),
\end{equation}
we will spend the rest of the proof establishing
\[
	\mathcal{N}_{X}^B(1_A, \ldots, 1_A) \geq \alpha^{k + 2} + O(\epsilon).
\]
From the decomposition $1_A = f_1 + f_2 + f_3$,
\[
	\mathcal{N}_{X}^B(1_A, \ldots, 1_A) = \sum_{1 \leq i_0, \ldots, i_{k + 1} \leq 3} \mathcal{N}_{X}^B(f_{i_0}, \ldots, f_{i_{k + 1}}),
\]
and we handle separately the three cases (i) when each of $i_j$ equals 1, (ii) when one of $i_j$ equals 2, and (iii) when one of $i_j$ equals 3.  Case (i) will yield our main term of $\alpha^{k + 2}$, and we will argue that cases (ii) and (iii) contribute a negligible $O(\epsilon)$ error.

For the first case, applying \eqref{bohrclosure2} and \eqref{bohrCountNormalized2} yields
\[
	\mathcal{N}_{X}^B(f_1, \ldots, f_1) = \E_{x_0} f_1(x_0)^{k + 2}\sigma(x) \E_{x_1, \ldots, x_k} \prod_{j = 1}^k \mu_{B_2} * \mu_{B_2}(x_j)\sigma^{(j + 1)}(x_j)  + O(\epsilon).
\]
Arguing as we did to establish \eqref{bohrCountNormalized1},
\[
	\mathcal{N}_{X}^B(f_1, f_1, f_1, f_1) = \E_x f_1(x)^{k + 2}\sigma(x) + O(\epsilon).
\]
Then from H\"{o}lder's inequality,
\begin{align*}
	\mathcal{N}_{X}^B(f_1, f_1, f_1, f_1) &\geq (\E_x f_1(x)\sigma(x))^{k + 2} + O(\epsilon)\\
	&= (\E_x 1_A(x) \nu_B * \sigma(x))^{k + 2} + O(\epsilon)\\
	&= \alpha^{k + 2} + O(\epsilon).
\end{align*}
For the second case, we apply Fourier inversion to write
\[
	\prod_{j = 1}^k \mu_{B_2}* \mu_{B_2}(x_j) = \sum_{\xi_1, \ldots, \xi_k \in \mathbf{F}_q^d} \prod_{j = 1}^k \widehat{\mu}_{B_2}(\xi_j)^2 \chi(-\xi_j \cdot x_j)
\]
in order to express
\[
	\mathcal{N}_{X}^B(f_{i_0},\ldots,f_{i_{k + 1}}) =  \sum_{\xi_1, \ldots, \xi_k \in \mathbf{F}_q^d} \prod_{j = 1}^k \widehat{\mu}_{B_2}(\xi_j)^2 \mathcal{N}_{X}(f_{i_0}\chi_{\xi_1 + \cdots + \xi_k}, f_{i_1}\chi_{-\xi_1}, \ldots f_{i_k}\chi_{-\xi_k}, f_{i_{k + 1}}).
\]
As before, $\norm{f\chi_\xi}_{U^2_\perp(S)} = \norm{f}_{U^2_\perp(S)}$, so we apply the triangle inequality, \cref{gvn} and Plancherel to conclude
\begin{align*}
	|\mathcal{N}_{X}^B(f_{i_0},\ldots,f_{i_{k + 1}})| &\leq \beta^{-k} \norm{f_2}_{U^2_\perp(S)} + O(\epsilon)\\
	&\leq \left(\frac{C|\Gamma|}{\epsilon\rho}\right)^{C|\Gamma|} \varphi(|\Gamma|^{-1},\rho) + O(\epsilon).
\end{align*}
We see that by taking
\[
	\varphi(|\Gamma|^{-1},\rho) = \epsilon\left(\frac{c\epsilon\rho}{|\Gamma|}\right)^{C|\Gamma|},
\]
we can ensure case (ii) contributes at most $O(\epsilon)$.

For the third case, we assume $i_0 = 3$ and apply the triangle inequality for
\[
	\mathcal{N}_{X}^B(f_3, f_{i_1},\ldots, f_{i_{k + 1}}) \leq 2^{k + 1} \E_{x_0} |f_3(x_0)| \sigma(x_0) \E_{x_1, \ldots, x_k} \prod_{j = 1}^k \mu_{B_2} * \mu_{B_2}(x_j) \sigma^{(j + 1)}(x_j) + O(\epsilon).
\]
Applying Cauchy-Schwarz and the fact that $\norm{f_3}_{L^2(S)} \leq \epsilon$, we conclude case (iii) again contributes at most an $O(\epsilon)$ error term.  In total we have shown
\[
	\mathcal{N}^B_{X}(1_A, \ldots, 1_A) \geq \alpha^{k + 2} + O(\epsilon)
\]
which we can insist is at least $\alpha^{k + 2}/2$ by taking $\epsilon$ sufficiently small with respect to $\alpha$ and $k$.  Then since $\beta \geq c(\alpha,k)$, we are done by recalling \eqref{triviallowerbound2}.
\end{proof}

It remains to prove both \cref{gvn} and \cref{decomposition}.

\section{Proof of \cref{gvn}}\label{gvnproofSection}

We establish \cref{gvn} through two careful applications of Cauchy-Schwarz.  In a sense, each application of Cauchy-Schwarz replaces our configuration with a more regular configuration, and we are left considering averages over geometric rectangles rather than more complicated spherical configurations.

%%%%% BEGIN ADAPTED DISSERTATION PROOF %%%%%
%\begin{proof}[Proof of \cref{gvn}]

We fix a spherical $(k + 2)$-point configuration $X$ spanning $k$ dimensions parameterized by
\[
	X = \{v_0, v_1, \ldots, v_{k + 1}\}
\]
with $v_1 - v_0, \ldots, v_k - v_0$ linearly independent and
\[
	v_{k + 1} = \sum_{j = 0}^{k} a_j v_j
\]
for coefficients $a_j \in \mathbf{F}_q$.  It is enough by symmetry to show
	\[
		|\mathcal{N}_{X}(f_0, \ldots, f_{k + 1})| \leq \norm{f_{k}}_{U^2_\perp(S)} + O(q^{-1/8}).
	\]
	
For $x_0, \ldots, x_k \in \mathbf{F}_q^d$, we will understand that
\[
	x_{k + 1} = \sum_{j = 0}^{k} a_ix_i,
\]
and we will detect if $\{x_0, \ldots, x_{k + 1}\} \simeq X$ with the spherical measures $\sigma(x_0)$ and, for $1 \leq j \leq k$,
\[
	\sigma^{(j)}(x_j) := \sigma^{c_{0,j}, \ldots, c_{j - 1,j}, \lambda}_{x_0, \ldots, x_{j - 1}}(x_j)
\]
where we define $c_{i,j} = |v_i - v_j|^2/2 - \lambda$, since if $|x_i|^2 = |x_j|^2 = \lambda$, checking whether $|x_i - x_j|^2 = |v_i - v_j|^2$ amounts to checking whether $x_i \cdot x_j = |v_i - v_j|^2/2 - \lambda$.  Then we can express our counting operator
\[
	\mathcal{N}_{X}(f_0, \ldots, f_{k + 1}) = \E_{x_0, \ldots, x_k} \sigma(x_0) \prod_{j = 0}^{k + 1} f_j(x_j) \prod_{j = 1}^{k} \sigma^{(j)}(x_j) + O(q^{-1}),
\]
where we have included the negligible amount of linearly dependent vectors.  Rearranging and applying the triangle inequality,
\[
	|\mathcal{N}_{X}(f_0, \ldots, f_{k + 1})| \leq \E_{x_0, \ldots, x_{k - 1}} \sigma(x_0)\prod_{j = 1}^{k - 1} \sigma^{(j)}(x_j) \left|\E_{x_k} f_{k}(x_{k}) f_{k + 1}(x_{k + 1}) \sigma^{(k)}(x_{k})\right| + O(q^{-1}).
\]
Introducing the differencing notation
\[
	\Delta_h f(x) := f(x)\overline{f}(x + h),
\]
we square both sides and apply Cauchy-Schwarz for
\[
	|\mathcal{N}_{X}(f_0, \ldots, f_{k + 1})|^2 \leq \E_{x_0, \ldots, x_{k - 1}} \sigma(x_0)\prod_{j = 1}^{k - 1} \sigma^{(j)}(x_j) \E_{x_k,h} \Delta_h f_{k}(x_{k}) \Delta_h f_{k + 1}(x_{k + 1}) \Delta_h \sigma^{(k)}(x_{k}) + O(q^{-1}).
\]
We introduce new spherical measures with an additional condition involving $h$ defined by
\begin{align*}
\sigma_h(x_0) &:= \sigma_h^0(x_0)\\
\sigma_h^{(j)}(x_j) &:= \sigma_{h,x_0,\ldots,x_{j - 1}}^{0,c_{0,j}, \ldots, c_{j - 1,j}, \lambda}(x_j).
\end{align*}
which allows us to rewrite
\[
	\sigma(x_0) \prod_{j = 1}^{k - 1} \sigma^{(j)}(x_j) \Delta_h \sigma^{(k)}(x_k) = \sigma_h(x_0) \prod_{j = 1}^{k - 1} \sigma_h^{(j)}(x_j) \sigma^{(k)}(x_k) \sigma(x_k + h).
\]
Our bound from above can then be rearranged as
\[
	|\mathcal{N}_{X}(f_0, \ldots, f_{k + 1})|^2 \leq \E_{h,x_0, \ldots, x_{k - 1}} \sigma_h(x_0)\prod_{j = 1}^{k - 1} \sigma_h^{(j)}(x_j) \E_{x_k} \Delta_h f_{k}(x_{k}) \Delta_h f_{k + 1}(x_{k + 1}) \sigma^{(k)}(x_{k})\sigma(x_k + h) + O(q^{-1}).
\]
We claim that in the average above, the dependence of $x_{k + 1}$ on $x_k$ is superficial.  That is, while
\[
	x_{k + 1} = \sum_{j = 0}^{k} a_jx_j,
\]
it must be the case that at least two coefficients $a_j$ are nonzero since we are working with a $(k + 2)$-point spherical configuration.  In particular, $a_j \neq 0$ for some $0 \leq j < k$, and for this distinguished $j$, we can reindex in $x_j$ in the average above to replace $x_{k + 1}$ with
\[
	x_{k + 1}' := \sum_{j = 0}^{k - 1} a_j'x_j,
\]
allowing us to rearrange our bound above as
\[
	|\mathcal{N}_{X}(f_0, \ldots, f_{k + 1})|^2 \leq \E_{h,x_0, \ldots, x_{k - 1}} \Delta_h f_{k + 1}(x_{k + 1}') \sigma_h(x_0)\prod_{j = 1}^{k - 1} \sigma_h^{(j)}(x_j) \E_{x_k} \Delta_h f_{k}(x_{k}) \sigma^{(k)}(x_{k})\sigma(x_k + h) + O(q^{-1}),
\]
where we may have needed to adjust the implicit scalars in our spherical measures.  This allows us to proceed as before by applying the triangle inequality for
\[
	|\mathcal{N}_{X}(f_0, \ldots, f_{k + 1})|^2 \leq \E_{h,x_0, \ldots, x_{k - 1}} \sigma_h(x_0)\prod_{j = 1}^{k - 1} \sigma_h^{(j)}(x_j) \left|\E_{x_k} \Delta_h f_{k}(x_{k}) \sigma^{(k)}(x_{k})\sigma(x_k + h)\right| + O(q^{-1}),
\]
and Cauchy-Schwarz once more for
\[
	|\mathcal{N}_{X}(f_0, \ldots, f_{k + 1})|^4 \leq \E_{h,x_0, \ldots, x_{k - 1}} \sigma_h(x_0)\prod_{j = 1}^{k - 1} \sigma_h^{(j)}(x_j) \E_{x_k,h'} \Delta_{h'} \Delta_h f_{k}(x_{k}) \Delta_{h'} \sigma^{(k)}(x_{k})\Delta_{h'}\sigma(x_k + h) + O(q^{-1}).
\]
We can again reorganize our spherical measures, rewriting
\[
\sigma_h(x_0)\prod_{j = 1}^{k - 1} \sigma_h^{(j)}(x_j) \Delta_{h'} \sigma^{(k)}(x_{k})\Delta_{h'}\sigma(x_k + h) = \Delta_{h'} \Delta_h \sigma(x_k) \sigma_{x_k, h, h'}(x_0) \prod_{j = 1}^{k - 1} \sigma_{x_k,h,h',x_0, \ldots, x_{j - 1}}(x_j)
\]
for appropriate implicit scalars.  Using this to rearrange our bound above, we have
\[
	|\mathcal{N}_{X}(f_0, \ldots, f_{k + 1})|^4 \leq \E_{x_k, h, h'} \Delta_{h'} \Delta_h f_k \sigma(x_k) \E_{x_0, \ldots, x_{k - 1}} \sigma_{x_k, h, h'}(x_0) \prod_{j = 1}^{k - 1} \sigma_{x_k,h,h',x_0, \ldots, x_{j - 1}}(x_j) + O(q^{-1}).
\]
Applying \cref{mainAsymptotic}, we have, uniformly in $x_k, h, h'$ when $d \geq 2k + 6$,
\[
	\E_{x_0, \ldots, x_{k - 1}} \sigma_{x_k, h, h'}(x_0) \prod_{j = 1}^{k - 1} \sigma_{x_k,h,h',x_0, \ldots, x_{j - 1}}(x_j) = 1 + O(q^{-1/2}),
\]
establishing
\[
|\mathcal{N}_{X}(f_0, \ldots, f_{k + 1})|^4 \leq \E_{x_k, h, h'} \Delta_{h'} \Delta_h f_k \sigma(x_k) + O(q^{-1/2})
\]
from which we see $|\mathcal{N}_{X}(f_0, \ldots, f_{k + 1})| \leq \norm{f_k}_{U^2_\perp(S)} + O(q^{-1/8})$ as required. \qed

\section{An Inverse Theorem and Proof of \cref{decomposition}}\label{last}

This section is dedicated to establishing \cref{decomposition}. We begin in \cref{inverseSection} by establishing \cref{inverse}, an \emph{inverse theorem} which reveals that functions with large $U^2_\perp(S)$ norm must exhibit Fourier bias. 

\subsection{An Inverse Theorem}\label{inverseSection}
 One way to see that functions with large $U^2_\perp(S)$ norm exhibit Fourier bias is to relate the $U^2_\perp(S)$ norm to the usual $U^2(\mathbf{F}_q^d)$ norm.  In order to do so, define for $v,w \in \mathbf{F}_q^d$ the normalized indicator function
\[
	\ell_{v=w} = \begin{cases} q^d & \text{if } v = w \\ 0 & \text{otherwise} \end{cases}
\]
Then one can express
\[
	\norm{f}_{U^2_\perp(S)}^4 = \E_{x,y,z,w} \overline{f}\sigma(x)f\sigma(y)f\sigma(z)\overline{f}\sigma(w) \ell_{x + w = y + z}
\]
Expanding $\ell_{x + w = y + z}$ via orthogonality,
\[
	\ell_{x + w = y + z} = \sum_{\xi \in \mathbf{F}_q^d} \chi(\xi(x + w - y - z)),
\]
providing the identity
\[
	\norm{f}_{U^2_\perp(S)}^4 = \sum_{\xi \in \mathbf{F}_q^d} |\widehat{f\sigma}(\xi)|^4.
\]
The right hand side is precisely $\norm{f \sigma}_{U^2(\mathbf{F}_q^d)}$.  It is tempting to here use Plancherel to bound $\norm{f}_{U^2_\perp(S)}^4$ above by $\sup_{\xi \in \mathbf{F}_q^d} |\widehat{f\sigma}(\xi)|^2 \norm{f\sigma}_{L^2}^2$, but this is not generally helpful since the $L^2$ term may grow with $q$.  By being a bit more careful, we establish the following inverse theorem.

\begin{thm}\label{inverse}
	Let $f : \mathbf{F}_q^d \rightarrow \mathbf{C}$ with $|f| \leq 1$.  If $d \geq 8$, then
	\[
		\norm{f}_{U^2_\perp(S)} \leq \sup_{\xi \in \mathbf{F}_q^d} \abs{\widehat{f \sigma}(\xi)}^{1/4} + O(q^{-1/32}).
	\]
\end{thm}

\begin{proof}
We apply absolute values and the triangle inequality for
\[
	\norm{f}_{U^2_\perp(S)}^4 \leq \E_{x} \sigma(x)\abs{\E_{y,z,w} \overline{f}\sigma(y)\overline{f}\sigma(z)f\sigma(w)\ell_{x = y + z - w}}
\]
Applying Cauchy-Schwarz,
\[
	\norm{f}_{U^2_\perp(S)}^8 \leq \left(\E_x \sigma(x)\right) \left(\E_x \sigma(x) \abs{\E_{y,z,w} \overline{f}\sigma(y)\overline{f}\sigma(z)f\sigma(w)\ell_{x = y + z - w}}^2\right).
\]
For any fixed $x$ and $y$, we have
\[
	\E_{z,w} \sigma(z)\sigma(w)\ell_{x = y + z - w} = \E_{z} \sigma_{x + y}(z)
\]
for some measure $\sigma_{x + y}$.  Then since $d \geq 4$, we can apply \eqref{sphere} and \cref{mainAsymptotic} for
\[
	\norm{f}_{U^2_\perp(S)}^8 \leq \E_x \sigma(x) \E_{\substack{y_1,z_1,w_1 \\ y_2,z_2,w_2}} \overline{f}\sigma(y_1)f\sigma(y_2)\overline{f}\sigma(z_1)f\sigma(z_2)f\sigma(w_1)\overline{f}\sigma(w_2)\ell_{x = y_1 + z_1 - w_1}\ell_{x = y_2 + z_2 - w_2} + O(q^{-1/2})
\]
Moving the average in $x$ inside and rearranging a bit, this simplifies to
\[
	\norm{f}_{U^2_\perp(S)}^8 \leq \E_{\substack{y_1,z_1,w_1 \\ y_2,z_2,w_2}} \overline{f}\sigma(y_1)f\sigma(y_2)\overline{f}\sigma(z_1)f\sigma(z_2)f\sigma(w_1)\overline{f}\sigma(w_2)\sigma(y_1 + z_1 - w_1)\ell_{w_2 - w_1 = y_2 - y_1 + z_2 - z_1} + O(q^{-1/2})
\]
Setting $f_1 = f$ and $f_2 = \overline{f}$, rearranging a bit more and applying the triangle inequality,
\[
	\norm{f}_{U^2_\perp(S)}^8 \leq \E_{w_1,w_2} \sigma(w_1)\sigma(w_2) \abs{\E_{y_1, y_2, z_1, z_2} \sigma(y_1 + z_1 - w_1) \ell_{w_2 - w_1 = y_2 - y_1 + z_2 - z_1} \prod_{j = 1,2} f_j\sigma(y_j)f_j\sigma(z_j)} + O(q^{-1/2})
\]
Setting $f_3 = f$ and $f_4 = \overline{f}$ and again requiring $d \geq 4$, we apply Cauchy-Schwarz as before and rearrange for
\begin{equation}\label{endofinverse}
	\norm{f}_{U^2_\perp(S)}^{16} \leq  \E_{\substack{y_1, \ldots, y_4 \\ z_1, \ldots, z_4}} \prod_{j = 1}^4 f_j\sigma(y_j)f_j\sigma(z_j) \ell_{y_2 - y_1 + z_2 - z_1 = y_4 - y_3 + z_4 - z_3} W_{\substack{y_1, \ldots, y_4 \\ z_1, \ldots, z_4}} + O(q^{-1/2})
\end{equation}
where
\[
W_{\substack{y_1, \ldots, y_4 \\ z_1, \ldots, z_4}} = \E_{w} \sigma(w)\sigma(y_2 - y_1 + z_2 - z_1 + w)\sigma(y_1 + z_1 - w)\sigma(y_3 + z_3 - w)
\]
Since $d \geq 4$, we can restrict the sum in \eqref{endofinverse} to only consider the terms when the vectors $y_2 - y_1 + z_2 - z_1, y_1 + z_1, y_3 + z_3$ are linearly independent at the cost of an error that can be absorbed in our current error of $O(q^{-1/2})$.  For these vectors, the function
\[
	\sigma'(w) = \sigma(w)\sigma(y_2 - y_1 + z_2 - z_1 + w)\sigma(y_1 + z_1 - w)\sigma(y_3 + z_3 - w)
\]
is a measure for which \cref{mainAsymptotic} applies, allowing us to conclude
\[
	W_{\substack{y_1, \ldots, y_4 \\ z_1, \ldots, z_4}} = \E_w \sigma'(w) = 1 + O(q^{-1/2}),
\]
valid for $d \geq 8$.  We have arrived at the estimate
\[
	\norm{f}_{U^2_\perp(S)}^{16} \leq \E_{\substack{y_1, \ldots, y_4 \\ z_1, \ldots, z_4}} \prod_{j = 1}^4 f_j\sigma(y_j)f_j\sigma(z_j) \ell_{y_2 - y_1 + z_2 - z_1 = y_4 - y_3 + z_4 - z_3} + O(q^{-1/2})
\]
Expanding $\ell_{y_2 - y_1 + z_2 - z_1 = y_4 - y_3 + z_4 - z_3}$ via orthogonality, we have
\[
	\ell_{y_2 - y_1 + z_2 - z_1 = y_4 - y_3 + z_4 - z_3} = \sum_{\xi \in \mathbf{F}_q^d} \chi(\xi \cdot (y_4 - y_3 - y_2 + y_1 + z_4 - z_3 - z_2 + z_1)),
\]
from which we have the identity
\[
	\E_{\substack{y_1, \ldots, y_4 \\ z_1, \ldots, z_4}} \prod_{j = 1}^4 f_j\sigma(y_j)f_j\sigma(z_j) \ell_{y_2 - y_1 + z_2 - z_1 = y_4 - y_3 + z_4 - z_3} = \sum_{\xi \in \mathbf{F}_q^d} |\widehat{f\sigma}(\xi)|^8
\]
It follows that we can conclude
\begin{align*}
	\norm{f}_{U^2_\perp(S)}^{16} &\leq \sum_{\xi \in \mathbf{F}_q^d} |\widehat{f\sigma}(\xi)|^8 + O(q^{-1/2})\\
	&\leq \sup_{\xi \in \mathbf{F}_q^d} |\widehat{f\sigma}(\xi)|^4 \sum_{\xi \in \mathbf{F}_q^d} |\widehat{f\sigma}(\xi)|^4 + O(q^{-1/2})\\
	&\leq \sup_{\xi \in \mathbf{F}_q^d} |\widehat{f\sigma}(\xi)|^4 + O(q^{-1/2}),
\end{align*}
where we have used both $|\widehat{f\sigma}(\xi)| \leq 1 + O(q^{-1/2})$ and $\sum_{\xi \in \mathbf{F}_q^d} |\widehat{f\sigma}(\xi)|^4 = \norm{f}_{U^2_\perp(S)}^4 \leq 1 + O(q^{-1/2})$.
\end{proof}

\subsection{Proof of \cref{decomposition}}\label{arproofSection}

In order to use the Fourier bias from \cref{inverse} to prove \cref{decomposition}, we will construct a sequence of Bohr sets, refining at each stage, until we arrive at $B' \prec_\epsilon B$ for which $\nu_{B'} * (f\sigma)$ and $\nu_B * (f\sigma)$ are close in an $L^2$ sense.  This will require a number of technical lemmas, the first of which indicates that these two convolutions are somewhat orthogonal.

\begin{lem}\label{pythagoras}
	Let $\epsilon \in (0,1)$, $B = B(\Gamma,\rho)$ a regular Bohr set with $B' \prec_\epsilon B$, and  $f : \mathbf{F}_q^d \rightarrow \mathbf{C}$ with $|f| \leq 1$.  If $d \geq 2$ and $q \geq q(|\Gamma|,\rho,\epsilon)$, then
	\[
		\norm{\nu_{B'} * (f\sigma)}_2^2 - \norm{\nu_B * (f\sigma)}_2^2 \geq \norm{\nu_{B'} * (f\sigma) - \nu_B * (f\sigma)}_2^2 + O(\epsilon).
	\]
\end{lem}
	\begin{proof}
		By expanding the square, $\norm{\nu_{B'} * (f\sigma) - \nu_B * (f\sigma)}_2^2$ is equal to
		\[
		\norm{\nu_{B'} * (f\sigma)}_2^2 - 2\E_x\nu_{B'} * (f\sigma)(x)\nu_B * (f\sigma)(x) + \norm{\nu_B * (f\sigma)}_2^2,
		\]
		so it suffices to show that
		\[
		\norm{\nu_B * (f\sigma)}_2^2 \leq \E_x\nu_{B'} * (f\sigma)(x)\nu_B * (f\sigma)(x) + O(\epsilon)
		\]
		Using the relationship $B' \prec_\epsilon B$, we can apply \cref{regularity} to see
		\begin{align*}
			\abs{\widehat{\mu}_{B}(\xi)} &= \abs{\E_{x \in B} \chi(x \cdot \xi)}\\
			&\leq \abs{\E_{x \in B} \E_{y \in B'} \chi((x + y) \cdot \xi)} + \epsilon\\
			&= \abs{\E_{x \in B} \chi(x \cdot \xi) \E_{y \in B'} \chi(y \cdot \xi)} + \epsilon\\
			&\leq \abs{\widehat{\mu}_{B'}(\xi)} + \epsilon.
		\end{align*}		
		in which case we apply Plancherel and this observation for
		\[
		\norm{\nu_B * (f\sigma)}_2^2 = \sum_{\xi \in \mathbf{F}_q^d} \abs{\widehat{\mu}_B(\xi)}^8 \abs{\widehat{f\sigma}(\xi)}^2 \leq \sum_{\xi \in \mathbf{F}_q^d} \abs{\widehat{\mu}_B(\xi)}^4 \left(\abs{\widehat{\mu}_{B'}(\xi)} + \epsilon\right)^4 \abs{\widehat{f\sigma}(\xi)}^2
		\]
	Writing $\beta = |B|/q^d$, we can apply Plancherel and \eqref{twoconvolutions} for
	\[
	\sum_{\xi \in \mathbf{F}_q^d} |\widehat{\mu_B}(\xi)|^4 |\widehat{f\sigma}(\xi)|^2 = \norm{\mu_B * \mu_B * f\sigma(x)}_2^2 \leq 1 + O(\beta^{-1}q^{-1/2}).
	\]
	Using the bound \eqref{bohrsetsize}, we see that we can take $q$ sufficiently large and apply Plancherel again to conclude
	\begin{align*}
		\norm{\nu_B * f\sigma}_2^2 &\leq \sum_{\xi \in \mathbf{F}_q^d} |\widehat{\mu}_B(\xi)|^4|\widehat{\mu}_{B'}(\xi)|^4 |\widehat{f\sigma}(\xi)|^2 + O(\epsilon)\\
		&= \E_x \nu_{B'} * (f\sigma)(x)\nu_B*(f\sigma)(x) + O(\epsilon).\qedhere
	\end{align*}
	\end{proof}

With \cref{pythagoras} in hand, we are ready to translate \cref{inverse} into an energy increment, showing that Fourier bias leads to a Bohr set refinement with increased $L^2$ energy.

\begin{prop}\label{energyincrement} Let $\eta \in (0,1)$, $\epsilon \in (0,c\eta^8)$ for $c > 0$ sufficiently small, $B = B(\Gamma,\rho)$ a regular Bohr set, and $f : \mathbf{F}_q^d \rightarrow \mathbf{C}$ with $|f| \leq 1$.  If $d \geq 2$, $q \geq q(|\Gamma|,\rho,\eta,\epsilon)$, and
\[
	\norm{f - \nu_B * (f\sigma)}_{U^2_\perp(S)} \geq \eta,
\]
then there exists $B' = B(\Gamma',\rho')$ with $B' \prec_{\epsilon} B$, $\abs{\Gamma'} \leq \abs{\Gamma} + 1$, $\rho' \geq c(|\Gamma|,\epsilon) \rho$ and
\[
	\norm{\nu_{B'} * (f\sigma)}_{L^2(S)}^2 \geq \norm{\nu_{B} * (f\sigma)}_{L^2(S)}^2 + c\eta^8.
\]
\end{prop}

\begin{proof}
	From the inverse theorem, $\norm{f - \nu_{B} * (f\sigma)}_{U^2_\perp(S)} \geq \eta$ implies the existence of some $\gamma \in \mathbf{F}_q^d$ with
	\[
		|\widehat{f\sigma}(\gamma) - [(\nu_{B} * (f\sigma))\sigma]^\wedge(\gamma)| \geq \eta^4 + O(q^{-1/32}).
	\]
	Expanding $[(\nu_{B} * (f\sigma))\sigma]^\wedge(\gamma)$, we can apply \eqref{sphere} for
	\begin{align*}
		[(\nu_B * (f\sigma))\sigma]^\wedge(\gamma) &= \sum_{\xi \in \mathbf{F}_q^d} \widehat{\sigma}(\xi)\widehat{\mu}_B(\gamma - \xi)^4 \widehat{f\sigma}(\gamma - \xi)\\
		&= \widehat{\mu}_B(\gamma)^4\widehat{f\sigma}(\gamma) + O\Big(q^{-1/2}\sum_{\gamma}|\widehat{\mu}_B(\xi)|^4|\widehat{f\sigma}(\xi)|\Big)
	\end{align*}
	Applying the uniform bounds $|\widehat{\mu}_B(\xi)| \leq 1$ and $\abs{\widehat{f\sigma}(\xi)} \leq 1 + O(q^{-1/2})$ together with Plancherel, we have
	\begin{equation}\label{sigmasigma}
	[(\nu_B * (f\sigma))\sigma]^\wedge(\gamma) = \widehat{\mu}_B(\gamma)^4\widehat{f\sigma}(\gamma) + O(\beta^{-1}q^{-1/2}).
	\end{equation}
	in which case we can ensure $q$ is sufficiently large and refine our inverse statement to
	\begin{equation}\label{refinedinverse}
		|\widehat{f\sigma}(\gamma)| |1 - \widehat{\mu}_B(\gamma)^4| \geq \eta^4/2.
	\end{equation}
	Set $\Gamma' = \Gamma \cup \{\gamma\}$.  It follows from \cite[Lemma 4.25]{taoVu} that for any $\epsilon \in (0,1)$ we can choose a radius $\rho' \in [\epsilon\rho / (400|\Gamma|), \epsilon\rho / (200|\Gamma|)]$ such that $B' := B(\Gamma',\rho')$ is regular, in which case $B' \prec_\epsilon B$.  From \cref{pythagoras},
	\[
		\norm{\nu_{B'} * (f\sigma)}_2^2 - \norm{\nu_B * (f\sigma)}_2^2 \geq \norm{\nu_{B'} * (f\sigma) - \nu_{B} * (f\sigma)}_2^2 + O(\epsilon),
	\]
	so we would like to provide a lower bound for
	\[
		\norm{\nu_{B'} * f\sigma - \nu_{B} * f\sigma}_2^2	= \sum_{\xi \in \mathbf{F}_q^d} \abs{\widehat{\mu}_{B'}(\xi)^4 - \widehat{\mu}_{B}(\xi)^4}^2\abs{\widehat{f\sigma}(\xi)}^2 \geq \abs{\widehat{\mu}_{B'}(\gamma)^4 - \widehat{\mu}_{B}(\gamma)^4}^2\abs{\widehat{f\sigma}(\gamma)}^2
	\]
	As $\gamma \in \Gamma'$, $\abs{\widehat{\mu}_{B'}(\gamma) - 1} \leq \rho' \leq \epsilon$.  Combining this with $|\widehat{f\sigma}(\gamma)| \leq 1 + O(q^{-1/2})$ and \eqref{refinedinverse},
	\[
		\abs{\widehat{\mu}_{B'}(\gamma)^4 - \widehat{\mu}_{B}(\gamma)^4}^2\abs{\widehat{f\sigma}(\gamma)}^2 \geq |1 - \widehat{\mu}_B(\gamma)^4|^2 |\widehat{f\sigma}(\gamma)|^2 + O(\epsilon) \geq \eta^8/4 + O(\epsilon).
	\]	
	Provided $\epsilon \leq c\eta^8$ for a sufficiently small absolute constant $c > 0$, we have managed to show
	\[
		\norm{\nu_{B'} * (f\sigma)}_2^2 - \norm{\nu_B * (f\sigma)}_2^2 \geq c\eta^8.
	\]
	For the same conclusion with $L^2(S)$ norms, we can apply \eqref{sigmasigma} for
	\begin{align*}
		\norm{\nu_B * (f\sigma)}_{L^2(S)}^2 &= \sum_{\xi \in \mathbf{F}_q^d} [\nu_B * (f\sigma)\sigma]^\wedge(\xi) \overline{\widehat{\mu}_B(\xi)^4 \widehat{f\sigma}(\xi)}\\
		&= \norm{\nu_B * (f\sigma)}_2^2 + O(\beta^{-2}q^{-1/2}).
	\end{align*}
	and ensure $q$ is sufficiently large.
\end{proof}

Next we iterate \cref{energyincrement} in order to find a Bohr set refinement for which $\nu_{B'} * (f\sigma)$ approximates $f$ well in a $U^2_\perp(S)$ sense.

\begin{prop}\label{kvn} Let $\eta \in (0,1)$, $\epsilon \in (0,c\eta^8)$ for $c > 0$ sufficiently small, $B = B(\Gamma,\rho)$ a regular Bohr set, and $f : \mathbf{F}_q^d \rightarrow \mathbf{C}$ with $|f| \leq 1$.  If $d \geq 2$ and $q \geq q(|\Gamma|,\rho,\eta,\epsilon)$, then there exists $B' = B(\Gamma', \rho')$ with $B' \prec_{\epsilon} B$, $\abs{\Gamma'} \leq C(\abs{\Gamma}, \eta)$, $\rho' \geq c(|\Gamma|, \eta, \epsilon) \rho$, and
	\[
		\norm{f - \nu_{B'} * (f \sigma)}_{U^2_\perp(S)} \leq \eta.
	\]
\end{prop}

\begin{proof}
	We will construct a sequence of Bohr sets
	\[
	B_{j + 1} \prec_{\epsilon} B_j \prec_{\epsilon} \cdots \prec_{\epsilon} B_1 \prec_{\epsilon} B_0
	\]
	where we write $B_j = B(\Gamma_j, \rho_j)$ and start with $B_0 = B$.  Our goal is find $j \geq 1$ for which
	\begin{equation}\label{kvnConclusion}
		\norm{f - \nu_{B_j} * (f \sigma)}_{U^2_\perp(S)} \leq \eta.
	\end{equation}
	We set $\Gamma_1 = \Gamma$ and choose $\rho_1 \in [\epsilon\rho/(400|\Gamma|),\epsilon\rho/(200|\Gamma|)]$ so that $B_1 \prec_{\epsilon} B_0$.  Starting with $j = 1$ and iterating, we are done if \eqref{kvnConclusion} is satisfied.  If \eqref{kvnConclusion} is false for $j \geq 1$, then by applying \cref{energyincrement} with $q$ sufficiently large, we obtain $B_{j + 1}$ for which $B_{j + 1} \prec_{\epsilon} B_j$, $\abs{\Gamma_{j + 1}} \leq \abs{\Gamma} + j$, $\rho_{j + 1} \geq c(\abs{\Gamma}, \eta, j)\rho$, and
	\[
		\norm{\nu_{B_{j + 1}} * (f\sigma)}_{L^2(S)}^2 \geq \norm{\nu_{B_{j}} * (f\sigma)}_{L^2(S)}^2 + c\eta^8.
	\]
	From \eqref{fourconvolutions}, we can ensure that that for every $j \geq 0$, we have
	\[
		\norm{\nu_{B_{j + 1}} * (f\sigma)}_{L^2(S)}^2 \leq 2,
	\]
	in which case it is clear that we can iterate this process at most $O(\eta^{-8})$ times before we arrive at some $j$ for which \eqref{kvnConclusion} is satisfied.
\end{proof}

Finally, we iterate \cref{kvn} to find a Bohr set $B$ for which $\nu_B * (f\sigma)$ approximates $f$ arbitrarily well in the $U^2_\perp(S)$ norm, at the cost of an $L^2$-error term.  This is necessary since shrinking $\eta$ in \cref{kvn} leads to a loss of control of the size of the approximating Bohr set.

%\begin{proof}[Proof of \cref{decomposition}]
	Letting $\epsilon > 0$, we will construct a sequence of Bohr sets
	\[
	B_{J} \prec_{\epsilon} \cdots \prec_{\epsilon} B_1 \prec_{\epsilon} B_0
	\]
	where we write $B_j = B(\Gamma_j,\rho_j)$.  We start with $\Gamma_0 = \{0\}$, $\rho_0 = 1$, and $\eta_0 = \varphi(1,1)$, in which case $B_0 = \mathbf{F}_q^d$.  For $1 \leq j \leq J$, we set $\eta_{j} = \varphi(\abs{\Gamma_{j-1}}^{-1},\rho_{j-1})$, ensure $\epsilon \in (0,c\eta_{j}^8)$, and apply \cref{kvn} to obtain $B_{j}$ with $B_{j} \prec_{\epsilon} B_{j - 1}$, $\abs{\Gamma_{j}} \leq C(|\Gamma_{j - 1}|, \eta_{j - 1})$ and $\rho_{j} \geq c(|\Gamma_{j - 1}|,\eta_{j - 1})$ such that
	\[
		\norm{f - \nu_{B_j} * (f\sigma)}_{U^2_\perp(S)} \leq \eta_j
	\]
	for sufficiently large $p$.  We claim there exists some $j$ for which
	\begin{equation}\label{l2claim}
		\norm{\nu_{B_{j + 1}} * (f\sigma) - \nu_{B_j} * (f\sigma)}_{L^2(S)}^2 \leq \eta^2.
	\end{equation}
	Using \cref{pythagoras}, we can insist $\epsilon$ is sufficiently small with respect to $\eta$ to guarantee
	\[
		\norm{\nu_{B_{j + 1}} * (f\sigma) - \nu_{B_j} * (f\sigma)}_{L^2(S)}^2 \leq \norm{\nu_{B_{j + 1}} * (f\sigma)}_{L^2(S)}^2 - \norm{\nu_{B_j} * (f\sigma)}_{L^2(S)}^2 + \eta^2/2.
	\]
	If $\norm{\nu_{B_{j + 1}} * (f\sigma)}_{L^2(S)}^2 \leq \norm{\nu_{B_j} * (f\sigma)}_{L^2(S)}^2$, then  \eqref{l2claim} would follow.  Suppose instead that \[
		\norm{\nu_{B_{j + 1}} * (f\sigma)}_{L^2(S)}^2 > \norm{\nu_{B_j} * (f\sigma)}_{L^2(S)}^2
		\]
		for all $j$.  From \eqref{fourconvolutions}, $\norm{\nu_{B_j} * (f\sigma)}_{L^2(S)}^2 \leq 2$ for sufficiently large $q$, in which case the $\norm{\nu_{B_j} * (f\sigma)}_{L^2(S)}^2$ form a bounded increasing sequence.  It follows that we can take $J = O(\eta^{-2})$ and conclude \eqref{l2claim} for some now fixed $1 \leq j \leq J$.  We are then finished by setting $B = B_j$ and
		\begin{align*}
			f_1 &= \nu_{B_j} * (f\sigma)\\
			f_2 &= f - \nu_{B_{j + 1}} * (f\sigma)\\
			f_3 &= \nu_{B_{j + 1}} * (f\sigma) - \nu_{B_j} * (f\sigma).
		\end{align*}
This completes the proof of \cref{decomposition}.\qed
%\end{proof}

%\smallskip

\section*{Appendix: Necessity of the Spherical Condition}

By an isometry of $\mathbf{F}_q^d$ we mean a linear map $U:\mathbf{F}_q^d\to\mathbf{F}_q^d$ so that $|U(x)|^2=|x|^2$ for all $x$. It is easy to see that $U(x)\cdot U(y)=x\cdot y$ for all $x$ and $y$, hence $U^TU=UU^T=I$. We call a set $X=\{x_0,\ldots,x_k\}$ \emph{non-degenerate} if $V\cap V^\bot=\{0\}$ for the subspace $V=\Span(X-X)$.

\begin{lem} Suppose $X$ is a non-degenerate $k+1$-point configuration and  $Y=\{y_0,\ldots,y_k\}$ is isometric to $X$. Then there exists a vector $z$ and an isometry $U$ of $\mathbf{F}_q^d$ so that $Y=z+U(X)$.
\end{lem}

\begin{proof} By performing a translation we can assume $x_0=y_0=0$. Assume that the vertices are labeled so that $|y_i-y_j|^2=|x_i-x_j|^2$ for all $0\leq i\leq j\leq k$.  Let $V$ and $W$ be the subspace spanned by the sets $X$ and $Y$ respectively. Since $X$ and hence $Y$ are non-degenerate, we have that the map $\phi:x_i\to y_i,\ 1\leq i\leq k$ extends to an isometry $\Phi:V\to W$. By Witt's extension theorem \cite{witt37} there is an orthogonal transformation $U:\mathbf{F}_q^d\to \mathbf{F}_q^d$ extending the map $\Phi$. Clearly $U(X)=Y$.
\end{proof}

Let us reformulate the property that a set $X$ is non-degenerate. Assume $X=\{0,x_1,\dots x_k\}$ and $X'=\{x_1,\ldots,x_k\}$ is a linearly independent set such that $\Span(X')=\Span(X)$. Let $M$ be the symmetric $l\times l$ matrix with entries $m_{ij}=x_i\cdot x_j$ for $1\leq i,\,j\leq l$.

\begin{lem} The set $X$ is non-degenerate if and only if the associated inner product matrix $M$ has maximal rank.
\end{lem}

\begin{proof} Clearly $X$ is non-degenerate if and only if $X'$ is linearly independent. Let $V= \Span (X')$ and assume that $0\neq v\in V$ such that $v\bot V$. Then $v=\sum_{i=1}^l b_i x_i$ and $v\cdot x_j=\sum_{i=1}^l m_{ji} b_i =  Mb_j=0$ for all $1\leq j\leq l$. Thus $Mb=0$ for a non-zero vector $b=(b_j)$ and $M$ has rank less than $l$. Conversely if $\rank(M)<l$ then $Mb=0$ for some $b=(b_j)\neq 0$ and the vector $v=\sum_i b_i x_i$ is orthogonal to $V=\Span(X')=\Span(X)$.
\end{proof}

Note that the inner product matrix drops rank if some non-trivial algebraic relations between the distances $|x_i-x_j|^2$ of the points $x_i,\,x_j$ of $X$, hence generic configurations are non-degenerate.  Next, we show that the spherical condition is necessary over finite fields as well, at least for non-degenerate sets. We follow the proof in \cite{egmrss} with some minor modifications. First we give the following characterization of spherical sets.

\begin{lem}\label{appendixspherical} Let $X=\{x_0,x_1,\ldots,x_k\}\subseteq \mathbf{F}_q^d$. Then $X$ is spherical if and only if the following holds. For every $c_0,c_1,\ldots,c_k\in \mathbf{F}_q$ if
\[(I)\quad\sum_{i=0}^k c_i=0,\ \ \textit{ and}\ \ \ \ \sum_{i=0}^k c_i x_i=0,\ \ \textit{then also}\]  \[(II)\quad\sum_{i=0}^k c_i |x_i|^2=0.\]
\end{lem}

\begin{proof} Suppose $X$ is spherical, that is $|x_i-z|^2=r$ for $0\leq i\leq k$ for some $z\in\mathbf{F}_q^d$ and $r\in \mathbf{F}_q$. Then $|x_i|^2=r-|z|^2+2x_i\cdot z$ for all $i$, hence
\[\sum_{i=0}^k c_i |x_i|^2 = \sum_{i=0}^k c_i (|z|^2+r) + 2\sum_{i=0}^k c_i x_i\cdot z =0.\]
Conversely, assume $X$ is not spherical. We show that there exists $c_0,\ldots,c_k$ satisfying (I) but not (II). One may assume that $X$ is minimal that is $X'$ is spherical for $X'\subseteq X$, $X'\neq X$. Since a simplex is spherical there is a non-trivial linear combination $\sum_{i=1}^k a_i (x_i-x_0)=0$ and by reindexing the vertices one may assume $a_k\neq 0$. Taking $X'=\{x_0,\ldots,x_{k-1}\}$ one has $|x_i-z|^2=r$ for $1\leq i<k$ but $|x_k-z|^2-r=b\neq 0$. Then $|x_i|^2-|x_0|^2= 2(x_i-x_0)\cdot z$ for $1\leq i<k$ and $|x_k|^2-|x_0|^2= 2(x_k-x_0)\cdot z\,+b$. Thus
\[\sum_{i=1}^k a_i (|x_i|^2-|x_0|^2) = 2 \sum_{i=1}^k a_k (x_i-x_0)\cdot z +a_k b=a_k b\neq 0.\]
Taking $c_0=-\sum_{i=1}^k a_i$ and $c_i=a_i$ for $1\leq i\leq k$ the claim follows.
\end{proof}

\begin{lem} Let $X\subseteq \mathbf{F}_q^d$ be a non-spherical set of $k$ elements. Then there exists a set $A\subseteq \mathbf{F}_q^d$ of size $|A|\geq c_k q^d$ which does not contain any set $Y$ of the form $Y=z+U(X)$, where $c_k>0$ is a constant depending only on $k$.
\end{lem}

\begin{proof} Let $t\in \mathbf{F}_q^\ast$ such that $|\chi(t)-1|\geq 1/2$, $\chi$ being a non-trivial character. Such a $t$ exists as $\sum_t \chi(t)=0$. If $X$ is not spherical then, by scaling, there exists $c_0,\ldots,c_k$ so that (I) holds and $\sum_{i=0} c_i |x_i|^2=t$. If $Y=z+U(X)$ then $\sum_{i=0}^k c_i y_i=0$ and moreover
\[\sum_{i=0}^k c_i |y_i|^2 = \sum_{i=0}^k c_i (|z|^2+2z\cdot U(x_i)+|x_i|^2)=t.\]
Let $B=B(\Gamma,\rho)$ be the Bohr set with $\Gamma=\{c_0,\ldots,c_k\}$ and $\rho=\frac{1}{4(k+1)}$. By the estimate \eqref{bohrsetsize}, $|B|\geq c_k q$, for some constant $c_k>0$ depending only on $k$. Define $A=\{x\in\mathbf{F}_q^d:\ |x|^2\in B\}$. Since the number of solutions to $|x|^2=b$ is at least $\frac{1}{2}q^{d-1}$ uniformly for $b\in \mathbf{F}_q$ for $q\geq q_0$, we have that $|A|\geq \frac{c_k}{2}q^d$. We show that $Y=z+U(x)$ for some $z\in \mathbf{F}_q^d$ and an orthogonal transformation $U$ then $Y\subsetneq A$. Indeed if $Y\subseteq A$ then $|\chi (c_i |y_i|^2)-1|\leq \frac{1}{4(k+1)}$ for $0\leq i\leq k+1$. It follows
\[|\chi(\sum_{i=0}^k c_i |y_i|^2)-1|=|\prod_{i=0}^k \chi(c_i|y_i|^2)-1|\leq \sum_{i=0}^k |\chi(c_i|y_i|^2)-1|\leq 1/4.\]
This implies $|\chi(t)-1|\leq 1/4$ contradicting our choice of $t$.
\end{proof}

\bigskip

%\bibliography{ff-spherical}{}
%\bibliographystyle{amsplain}

\providecommand{\bysame}{\leavevmode\hbox to3em{\hrulefill}\thinspace}
\providecommand{\MR}{\relax\ifhmode\unskip\space\fi MR }
\providecommand{\MRhref}[2]{%
  \href{http://www.ams.org/mathscinet-getitem?mr=#1}{#2}
}
\providecommand{\href}[2]{#2}

\end{document}